\documentclass[11pt,twoside,a4paper]{amsart}
\usepackage[english,frenchb]{babel}
\usepackage{amssymb}
\usepackage{amsmath}
\usepackage{amsthm}
\usepackage[dvips]{graphics}
\usepackage{graphicx}
\usepackage{latexsym}
\usepackage[normal]{subfigure}
\input{epsf.sty}
\include{psfig}
\usepackage{eufrak}
\usepackage{mathrsfs}
\usepackage{lipsum}
\usepackage{mdframed}
\usepackage{perpage}
\usepackage{amsfonts}
\begin{document}
\def\K{\mathbb{K}}
\def\R{\mathbb{R}}
\def\C{\mathbb{C}}
\def\Z{\mathbb{Z}}
\def\Q{\mathbb{Q}}
\def\D{\mathbb{D}}
\def\N{\mathbb{N}}
\def\T{\mathbb{T}}
\def\P{\mathbb{P}}
\renewcommand{\theequation}{\thesection.\arabic{equation}}
\newtheorem{theorem}{Th\'eor\`eme}[section]
\newtheorem{cond}{C}
\newtheorem{lemma}{Lemme}[section]
\newtheorem{corollary}{Corollaire}[section]
\newtheorem{prop}{Proposition}[section]
\newtheorem{definition}{D\'efinition}[section]
\newtheorem{remark}{Remarque}[section]
\newtheorem{example}{Exemple}[section]
\bibliographystyle{plain}

\title{Approches courantielles  \`a la Mellin dans un cadre non archim\'edien}

\author {Ibrahima Hamidine}

\address{D\'epartement de Math\'ematiques\\UFR  Sciences et Technologies \\ Universit\'e Assane Seck de Ziguinchor, BP: 523 (S\'en\'egal)}

\email{i.hamidine5818@zig.univ.sn}

\thanks{This work was Supported by NLAGA-Simons project, partially supported by IMB of Bordeaux (France) and UASZ (S\'en\'egal).}

\date{\today}

\maketitle

\begin{abstract}
On propose une approche du type Mellin pour l'approximation des courants d'int\'egration ou la r\'ealisation 
effective de courants de Green normalis\'es associ\'es \`a un cycle $\bigwedge_1^m [{\rm div}(s_j)]$, 
o\`u $s_j$ est une section m\'eromorphe d'un fibr\'e en droites $\mathscr{L}_j \rightarrow U$ au-dessus 
d'un ouvert $U$ d'un bon espace de Berkovich, 
lorsque 
chaque $\mathscr{L}_j$ est \'equip\'e d'une m\'etrique lisse et
que ${\rm codim}_{U}\big(\bigcap_{j\in J}{\rm Supp}[{\rm div(s_j)}]\big)
\geq \# J$ pour tout ensemble $J\subset \{1,...,p\}$. On \'etudie aussi la transposition au cadre non archim\'edien 
des formules de Crofton et de King, en particulier la r\'ealisation approch\'ee de courants de Vogel et de Segre.
\vskip 1mm
\noindent
\textit{\bf Mots clefs.} Courants, diviseurs, \'equations de Lelong-Poincar\'e, formule de King, nombres de Lelong.
\vskip 2mm
\noindent
{\rm Abstract.} ({\bf Mellin's Currential approaches  in a  non archimedean framework}). We propose an approach of 
Mellin type for the approximation of integration currents or the effective realization of normalized Green currents 
associated with a cycle $ \bigwedge_1^m[{\rm div} (s_j)] $, where $s_j $  is a meromorphic section of a line bundle 
$ \mathscr{L}_j \rightarrow U$ over an open $U$ in a good Berkovich space  
 when each 
$ \mathscr{L}_j$ has a smooth  metric and  
$ {\rm codim}_{U}\big (\bigcap_{j \in J} {\rm Supp} [{\rm div (s_j)}] \big)\geq \# J$ for every set 
$ J \subset \{1, ..., p \} $. We also study the transposition to the non archimedean context of Crofton  and King formulas, 
particularly the approximate realization of Vogel and Segre currents.
\end{abstract}

\section*{Introduction}
\'Etant donn\'ee une fonction $f$ holomorphe dans un ouvert 
$\mathscr U$ de $\C^n$ et une $(n-1,n-1)$-forme diff\'erentielle 
$\varphi$ \`a support compact dans $\mathscr {U}$, la transform\'ee de Mellin 
de la fonction 
$$
\varepsilon \in\, ]0,+\infty[\, \longmapsto I_f(\varphi)(\varepsilon) = \frac{1}{2i\pi} \int_{\{|f|^2 = \varepsilon\}} 
\frac{df\wedge \varphi}{f}  
$$
est formellement la fonction m\'eromorphe  
\begin{equation*} 
\lambda \in \C\longmapsto \lambda \int_0^\infty \varepsilon^{\lambda-1} 
\, I_f(\varphi)(\varepsilon)\, d\varepsilon 
= \frac{1}{2i\pi} \, \Big\langle 
[\lambda |f|^{2(\lambda-1)} \, \overline{df}]\,,\, df\wedge \varphi \Big\rangle
= \Big\langle dd^c \Big[\frac{|f|^{2\lambda}}{\lambda}\Big]\,,\, \varphi \Big\rangle, 
\end{equation*} 
($dd^c=\frac{-\partial\overline{\partial}}{2i\pi}$) et sa valeur en $\lambda =0$ (qui n'est pas un p\^ole) est exactement 
$\langle [{\rm div}(f)]\,,\, \varphi\rangle$, ce qui constitue une mani\`ere alternative de formuler dans le cadre de la g\'eom\'etrie analytique complexe la formule 
de Lelong-Poincar\'e. Introduite pour la premi\`ere fois dans un tel contexte 
dans \cite[chapitre 1, sections 1 et 2]{BGVY}, cette approche s'est r\'ev\'el\'ee particuli\`erement utile pour envisager les probl\`emes d'intersection et de division dans le cadre de la g\'eom\'etrie analytique complexe en profitant du scindage du courant d'int\'egration $[{\rm div}(f)]$ en $[{\rm div}(f)] = \bar\partial [1/f] \wedge df$, o\`u $[1/f]$ est le $(0,0)$-courant Valeur Principale extension standard depuis l'ouvert $\mathscr U_f= \mathscr U \setminus 
{\rm Supp}([{\rm div}(f)])$ de la fonction holomorphe inversible $1/f$. 
Quand bien m\^eme pareil scindage du courant $[{\rm div}(f)]$ ne s'av\`ere plus possible dans le contexte de la g\'eom\'etrie analytique sur un corps $\K$ \'equip\'e d'une valeur  
absolue non archim\'edienne, un des objectifs de ce travail est de montrer qu'une approche du m\^eme type peut n\'eanmoins \^etre conduite dans ce cadre, suivant en particulier l'approche \`a la th\'eorie de l'intersection incompl\`ete 
propos\'ee dans \cite{ASWY14} (toujours dans le cadre de la g\'eom\'etrie analytique complexe).        
\vskip 1mm
\noindent
Soit $\mathbb K$ un corps \'equip\'e d'une valeur absolue ultram\'etrique (triviale ou non)
$|\ |$ pour laquelle $\mathbb K$ est aussi complet et $\Gamma$ son groupe des valeurs, c'est-\`a-dire le sous-groupe de
$\R$ obtenu comme l'image de $\K^*$ par la valuation $-\log |\ |$.
Soit $X$ une vari\'et\'e alg\'ebrique de dimension $n$ d\'efinie sur $\mathbb K$. On note $X^{\rm an}$ l'analytifi\'e de
$X$ au sens de Berkovich \cite[Remarque 3.4.2]{Berk90}, ensemblistement d\'ecrit ainsi (voir par exemple \cite[section 4]{Gub14})~: si $U ={\rm Spec}(A)$
est un ouvert affine de $X$, on consid\`ere l'ensemble $U^{\rm an}$ de toutes les semi-normes multiplicatives 
ultram\'etriques sur l'alg\`ebre $A$ 
prolongeant la valeur absolue $|\ |$ sur $\mathbb K$ et l'on \'equipe cet ensemble $U^{\rm an}$ de la topologie 
la moins fine rendant continues toutes 
les applications d'\'evaluation $x\in U^{\rm an}
\mapsto x(a)=|a|_x$ ($a\in A$)~; en recollant les divers $U^{\rm an}$ pour tous les ouverts affines $U$ de $X$, 
on obtient un espace topologique connexe, 
localement compact et s\'epar\'e, sur lequel il reste \`a construire un faisceau structural $\mathcal O_{X^{\rm an}}$ 
de fonctions dites r\'eguli\`eres. 
Ce faisceau est d\'efini ainsi (voir par exemple \cite[chapitre 1, section 1.2]{BPS} pour une pr\'esentation rapide). 
Chaque point
$x$ de $U^{\rm an}$ induit une norme sur l'anneau int\`egre
$B_x = A/{\rm ker}\, x$, o\`u ${\rm ker}\, x = \{a\in A\,;\, x(a) =0\}$,
cette norme s'\'etendant en une norme $|\ |_x$ sur le corps des fractions de
$B_x$, corps que l'on compl\`ete relativement \`a cette norme en un corps $\mathscr H_x$. Si
$\mathscr U$ est un ouvert de $U^{\rm an}$, une fonction r\'eguli\`ere dans
$\mathscr U$ (c'est-\`a-dire un \'el\'ement de $\mathcal O_{X^{\rm an}}(\mathscr{U})$) est par d\'efinition une fonction
$$
x\in \mathscr{U} \longmapsto f(x) \in \bigsqcup_{x\in \mathscr{U}} \mathscr{H}_x
$$
se pliant \`a la clause d'approximation suivante~: pour tout $x\in \mathscr{U}$, il existe un voisinage
$\mathscr{U}_x$ de $x$ dans $\mathscr{U}$ tel que
\begin{equation}\label{faisceau}
\begin{split}
& \forall\, \epsilon>0,\ \exists\, a_{x,\epsilon}
\in A,\, \exists\, b_{x,\epsilon}
\in A\setminus {\rm ker}\, x,\ {\rm tels\ que} \\
& \forall\, \xi\in \mathscr{U}_x\,,\ |f(\xi) - a_{x,\epsilon}(\xi)/b_{x,\epsilon}(\xi)|_{\xi } < \epsilon,
\end{split}
\end{equation}
o\`u il faut ici entendre $a_{x,\epsilon}(\xi)$ et $b_{x,\epsilon}(\xi)$
comme les classes dans $B_\xi$ des
\'el\'ements $a_{x,\epsilon}(\xi)=\xi(a_{x,\epsilon})$ et
$b_{x,\epsilon}(\xi)=\xi(b_{x,\epsilon})$, le quotient $a_{x,\epsilon}(\xi)/b_{x,\epsilon}(\xi)$
\'etant alors un \'el\'ement du corps des fractions de cet anneau int\`egre
$B_\xi$, donc de son compl\'et\'e $\mathscr H_\xi$ pour la norme $|\ |_\xi$.
\vskip 1mm
\noindent
Lorsque $X$ est le tore $T_r :=
{\rm Spec}\, \K[X_1^{\pm 1},...,X_r^{\pm 1}]$ (plus g\'en\'eralement
$\K[M_r]$, o\`u $M_r$ est un groupe ab\'elien libre de rang $r$), l'analytifi\'e
$T^{\rm an}_r$ se visualise ensemblistement ainsi~: se donner une
semi-norme multiplicative sur $\K[X_1^{\pm 1},...,X_r^{\pm 1}]$ prolongeant
la valeur absolue $|\ |$ sur $\mathbb K$ revient \`a se donner une $\K$-extension
$\mathbb K\subset \mathbb{L}$ avec une valeur absolue $|\ |_{\mathbb L}$ prolongeant la valeur absolue sur $\K$, 
un point $\ell \in \mathbb L^r$ et \`a poser
$x_{\mathbb L,\ell}(a)=|a(\ell)|_{\mathbb L}$.
On dispose d'une application continue et propre de tropicalisation~:
$$
x \in T_r^{\rm an} \stackrel{\rm trop}{\longmapsto} \big(- \log (x(X_1)),...,-\log (x(X_r))\big) \in \R^r\; 
({\rm ou}\ N_{r,\R},\ N_r = {\rm Hom}(M_r,\Z)).
$$
Cette application de tropicalisation admet la section suivante~: \`a $(\omega_1,...,\omega_r)
\in \R^r$ (respectivement dans $N_{r,\R}$), on associe la norme multiplicative sur $\K[X_1^{\pm 1},...,X_n^{\pm 1}]$ 
(respectivement sur $\K[M_r]$) d\'efinie par
$$
x_\omega (a) := \sup_{\alpha \in {\rm Supp}(a)} |a_\alpha|\, e^{-\langle \alpha,\omega\rangle}.
$$
L'image de $\R^r$ (ou plus g\'en\'eralement de $N_{r,\R}$) par cette section $\omega \mapsto x_\omega$ est appel\'ee 
squelette de $T^{\rm an}_r$~; 
c'est un sous-ensemble ferm\'e de $T^{\rm an}_r$ sur lequel $T_r^{\rm an}$ se r\'etracte continument au sens fort.
\vskip 2mm
\noindent
Le r\^ole des cartes locales sur $X^{\rm an}$ est tenu par les analytifications
des applications moment de la vari\'et\'e alg\'ebrique $X$ (voir par exemple \cite{Gub14}).
\vskip 2mm
\noindent
On dispose aussi sur $X^{\rm an}$ d'\^etres cette
fois \og souples\fg\ et non plus \og rigides\fg\
(comme le sont les sections locales du faisceau structurant), \`a savoir des sections du faisceau $\mathscr{A}^{0,0}$ 
des fonctions lisses. Si $U$ est un ouvert de
$X^{\rm an}$, une section de $\mathscr{A}^{0,0}$ dans $U$ est par d\'efinition une
fonction de $U$ dans $\R$ pouvant s'exprimer localement (au voisinage de tout point $x$ de $U$)
sous la forme
\begin{equation}\label{reguliere}
\xi \longmapsto \phi( \log |f_{x,1}|_\xi,...,\log |f_{x,m_x}|_\xi),
\end{equation}
o\`u les fonctions $f_{x,j}$ sont des fonctions r\'eguli\`eres inversibles au voisinage de
$x$ et $\phi$ une fonction de $\R^{m_x}$ dans $\R$ de classe $C^\infty$ au voisinage de l'image
dans $\R^{m_x}$ d'un voisinage de $x$ par l'application
$$
\xi \longmapsto \big(\log |f_{x,1}|_\xi,...,\log |f_{x,m_x}|_\xi\big) \in \R^{m_x}.
$$
Tel est le cas par exemple des fonctions s'exprimant au voisinage $U_x$ de tout point $x$ de $U$ sous la forme
\begin{equation*}
\xi \longmapsto \prod\limits_{j=1}^{m_x} \Big(\frac{|f_{x,j}(\xi)|_\xi}{e^{\rho_{x,j}}}\Big)
^{\lambda_j}
= \prod\limits_{j=1}^{m_x} \exp \big(\lambda_j\, (\log |f_{x,j}(\xi)|_\xi -\rho_{x,j})\big)
\end{equation*}
o\`u $\lambda_1,...,\lambda_{m_x}$ sont des param\`etres r\'eels, $f_{x,1},...,f_{x,m_x}$ des fonctions r\'eguli\`eres 
et inversibles au voisinage de $x$ et
$\rho_{x,1},...,\rho_{x,m_x}$ des \'el\'ements de $\mathscr{A}^{0,0}(U_x)$.
\vskip 2mm
\noindent
On dispose aussi sur $X^{\rm an}$, pour tout couple d'entiers $p,q$ entre $0$ et $n$, du faisceau $\mathscr{A}^{p,q}$ 
des $(p,q)$-formes et, par dualit\'e, du 
faisceau $\mathscr{D}_{p,q}$ des $(p,q)$-courants,
tous deux introduits par A. Chambert-Loir et A. Ducros (\cite{ChLD,ChL}) et W. Gubler \cite{Gub14}). Rappelons bri\`evement comment sont 
d\'efinis ces faisceaux. Afin de pouvoir profiter, 
tout en travaillant dans un cadre r\'eel, de la positivit\'e inh\'erente au cadre complexe, on introduit dans un 
premier temps sur $\R^r$ (ou $\N_{r,\R}$) le faisceau 
des $(p,q)$ super-formes ($0\leq p,q\leq r$)~: une section de ce faisceau dans un ouvert $\Omega\subset \R^r$ (ou 
$N_{r,\R}$) est une $(p,q)$-forme (au sens usuel) dans 
l'ouvert ${\rm Log}^{-1}(\Omega)\subset (\C^*)^r$ (ou $\C[N_r](\C)$) s'exprimant comme
$$
\omega = \sum\limits_{\stackrel{1\leq i_1 < \dots < i_p\leq r}
{1\leq j_1 < \dots < j_q\leq r}} \omega_{I,J} ({\rm Log}(z))\, \frac{dz_I}{z_I} \wedge
\frac{d\bar z_J}{\bar z_J},
$$
o\`u ${\rm Log}~: z \mapsto (\log |z_1|,...,\log |z_r|)$ et les $\omega_{I,J}$
sont des fonctions \`a valeurs r\'eelles de classe $C^\infty$ dans $\Omega$.
Par dualit\'e, on dispose du faisceau des super-courants de bidimension $(p,q)$, dont les sections
dans un ouvert $\Omega\subset \R^r$ (ou $N_{r,\R}$) sont les courants de bidimension $(p,q)$ dans ${\rm Log}^{-1}(\Omega)$ 
s'exprimant sous la forme
$$
T = \sum\limits_{\stackrel{1\leq i_1' < \dots < i_{n-p}'\leq r}
{1\leq j_1' < \dots < j_{n-q}'\leq r}} T_{I',J'} ({\rm Log}(z))\, \frac{dz_{I'}}{z_{I'}} \wedge
\frac{d\bar z_{J'}}{\bar z_{J'}},
$$
o\`u les distributions $T_{I',J'}$ sont des distributions r\'eelles dans $\Omega$ et
$$
\langle T_{I',J'}({\rm Log}(z)),\varphi \rangle :=
\Big\langle T_{I',J'}(x), \frac{1}{(2\pi)^r}\int_{[0,2\pi]^r}
\varphi(e^{x+i\theta})\, d\theta_1 \dots d\theta_r\Big\rangle
$$
pour toute fonction test $\varphi\in \mathcal D({\rm Log}^{-1}(\Omega),\R)$.
Le courant tropical de bidimension $(n,n)$ attach\'e (voir \cite{Bab,BabH}) \`a un cycle tropical de dimension $n<r$ dans 
$\R^r$ ou $N_{r,\R}$ (la pond\'eration \'etant 
prise en compte) constitue un exemple important de super-courant de bidimension $(n,n)$ dans $\R^r$ ou $N_{r,\R}$.
Pour plus de d\'etails sur la mani\`ere  dont les formes diff\'erentielles r\'eelles sur les espaces de Berkovich sont introduites,  voir \cite{Gub14}.
Le faisceau $\mathscr D_{p,q}(U)$ des courants de bidimension $(p,q)$ sur $X^{\rm an}$ est d\'efini comme le dual 
 du faisceau $\mathscr A^{p,q}(U)$ des $(p,q)$-formes diff\'erentielles (on notera par la suite, pour tout ouvert 
 $U$ de $X^{\rm an}$ par $\mathscr{A}^{p,q}_c(U)$ le $\R$-espace vectoriel des $(p,q)$-formes diff\'erentielles \`a support compact inclus dans $U$)
 \cite[section 6]{Gub14}. On dispose en particulier du courant de bidegr\'e $(p,p)$
d'int\'egration sur une sous-vari\'et\'e $Y$ de $X$ de codimension $0\leq p\leq n$ et, lorsque
$f$ est une fonction r\'eguli\`ere dans un ouvert $U$ de $X^{\rm an}$, du courant de bidegr\'e $(1,1)$ d'int\'egration 
(avec multiplicit\'es prises en compte) $[{\rm div}(f)]$. Si $u~: U \rightarrow \R$ est une fonction continue 
dans un ouvert $U$ de $X^{\rm an}$, $u$ d\'efinit naturellement un courant de bidegr\'e
$(0,0)$ dans $U$ (voir par exemple \cite[section 6]{GuK}), not\'e $[u]$.
\vskip 2mm
\noindent
Le but de ce travail est de d\'evelopper une approche bas\'ee sur le prolongement analytique aux fins de la r\'esolution 
de l'\'equation de Lelong-Poincar\'e, 
de la construction de courants de Green (voir la section \ref{sectionGreen}) et de l'explicitation des courants de Vogel 
(section \ref{sectionVogel}) ou de Segre 
(section \ref{sectionsegre})~; une interpr\'etation de la formule de King, en relation avec la d\'ecomposition de Siu des 
courants positifs ferm\'es et le 
couplage \og composantes stables/composantes mobiles\fg\ en th\'eorie de l'intersection impropre \cite{ASWY14} est aussi \'evoqu\'ee 
dans ce travail (section \ref{sectionKing}).
L'objectif vis\'e ici est d'esquisser une transposition au cadre non archim\'edien de l'approche \`a la th\'eorie de 
l'intersection dans le cadre impropre telle 
qu'elle est d\'evelopp\'ee dans  \cite{ASWY14}
(dans le contexte analytique complexe). \`A plus long terme, nous esp\'erons exploiter cette approche en la combinant avec 
le calcul d'Igusa \cite{Igu} (dans 
le contexte ultam\'etrique) aux fins d'unifier les contributions aux places finies et infinies dans par exemple l'expression 
de la hauteur logarithmique 
(relativement \`a un choix de m\'etrique, lisse ou non) d'un cycle arithm\'etique \cite{BPS}.

\section{Approche du type Mellin aux \'equations de Le\-long-Poincar\'e}\label{section1} 
Soit $X$ un bon $\K$-espace de Berkovich, c'est-\`a-dire un $\K$-espace de Berkovich dont tout point admet un voisinage 
affino\"ide, donc une base de voisinages 
affino\"ides. On suppose $X$ de dimension pure  $n$ et sans bord. On pourra  penser
$X$ comme l'analytifi\'e d'une  vari\'et\'e alg\'ebrique de dimension $n$ d\'efinie au-dessus de $\K$.
\vskip 1mm
\noindent
Soit $U\subset X$ un ouvert de $X$ et $f$ une fonction m\'eromorphe r\'eguli\`ere non diviseur de z\'ero dans $U$, c'est-\`a-dire une 
fonction s'exprimant au voisinage 
de tout point $x\in U$ comme le quotient de deux sections locales de $\mathcal O_{X}$.
On note $U_f$ le plus grand ouvert de $U$ dans lequel $f$ s'exprime localement comme
une section locale inversible du faisceau $\mathcal O_{X}$.
\vskip 1mm
\noindent
Si $\omega$ est une section \`a support compact de $\mathscr{A}^{n-1,n-1}_c(U)$,
le support (compact)
de la  $(n-1,n)$-forme $d''\omega$ \'evite tout sous-ensemble ferm\'e de Zariski d'int\'erieur vide 
\cite[lemme 3.2.5]{ChLD}. Pour tout
$\lambda\in \R^*$, la forme
$$
x\in U_f \mapsto \Big(d'\Big(\frac{|f|^\lambda}{\lambda}\Big) \wedge d''\omega\Big)(x) =
\Big(d'\Big(\frac{e^{\lambda \log|f|}}{\lambda}\Big)\wedge
d''\omega\Big)(x) 
$$
appartient \`a $\mathscr A^{n,n}_c(U_f)$ et l'on a d'apr\`es la formule de Stokes ($X$ est suppos\'ee sans bord), si $\varphi$ 
d\'esigne une fonction lisse 
identiquement \'egale \`a $1$ au voisinage du support de $d''\omega$ et de support compact dans $U_f$ 
(que l'on peut encore construire gr\^ace au th\'eor\`eme de partitionnement de l'unit\'e,  \cite[proposition 3.3.6]{ChLD}),
\begin{multline*} 
\int_{U_f} d'\Big(\frac{|f|^{\lambda}}{\lambda}\Big) \wedge d'' \omega 
= \int_{X} d'\Big(\varphi \frac{|f|^{\lambda}}{\lambda}\Big) \wedge d'' \omega \\ 
= -\int_{X} \Big(\varphi \, \frac{|f|^{\lambda}}{\lambda}\Big)\, d'd''\omega = 
- \Big\langle \Big[ 
\frac{|f|^\lambda}{\lambda}\Big]\,,\, d'd''\omega \Big\rangle 
= \Big\langle d'\, \Big[ 
\frac{|f|^\lambda}{\lambda}\Big]\,,\, d''\omega \Big\rangle, 
\end{multline*} 
o\`u le courant $[|f|^\lambda/\lambda]$ est le $(0,0)$-courant d\'efini \`a partir de la fonction lisse 
$$
|f|^{\lambda}/\lambda~: U_f \rightarrow \R
$$
suivant le lemme 4.6.1 de \cite{ChLD}.  
\vskip 1mm
\noindent     
On d\'efinit, pour tout $\lambda \in \R^*$, un courant
$T_\lambda^f \in \mathscr{D}_{1,1}(U)$ en posant
\begin{multline*}
\forall\, \omega \in \mathscr{A}^{n-1,n-1}_c(U),\quad
\langle T_\lambda^f,\omega\rangle := -
\Big\langle d'\, \Big[ 
\frac{|f|^\lambda}{\lambda}\Big]\,,\, d''\omega \Big\rangle = 
- \int_{U_f} d'\Big(\frac{|f|^{\lambda}}{\lambda}\Big) \wedge d'' \omega  \\ 
= -\int_{U_f} |f|^\lambda d'\Big(\log |f|\Big)\wedge d''\omega. 
\end{multline*}   
D'apr\`es le th\'eor\`eme de convergence domin\'ee de Lebesgue, la fonction 
$\lambda \mapsto T_\lambda^f$ (\`a valeurs dans $\mathscr D_{1,1}(U)$) 
admet comme limite lorsque $\lambda$ tend vers $0$ le $(1,1)$-courant 
\begin{multline*} 
\omega \in \mathscr A^{n-1,n-1}_c(U) 
\longmapsto - \int_{U_f} 
d'\Big(\varphi\, \log |f|\Big) \wedge d''\omega = 
\int_{X^{\rm an}} \varphi\, \log |f|\, d'd''\omega \\ 
= \big\langle \big[\varphi \, \log |f|\big]\,,\, d'd''\omega \big \rangle  
= \big\langle d'd''\, \big[\varphi \, \log |f|\big]\,,\,\omega\big\rangle = \big \langle [{\rm div}(f)]\,,\,\omega\big\rangle 
\end{multline*} 
d'apr\`es la formule de Stokes, une nouvelle fois le lemme    
4.6.1 de \cite{ChLD}, et enfin la formule de Lelong-Poincar\'e \cite[th\'eor\`eme 4.6.5 ]{ChLD}. 
\vskip 2mm
\noindent
On peut maintenant envisager le cas o\`u $f_1$ et $f_2$ sont deux fonctions r\'eguli\`eres m\'eromorphes dans un ouvert $U$ de $X$, telles que ${\rm codim}_{U}
({\rm Supp}\,([{\rm div}( f_1)]) \cap {\rm Supp}\,([{\rm div}( f_2)]))\geq 2$. Suivant la description du courant 
$[{\rm div}(f_1)]$ donn\'ee dans la section 4.6 de \cite[commentaire apr\`es la preuve du lemme 4.6.4]{ChLD}, on exprime 
ce courant comme la somme de courants d'int\'egration $\pm [{\rm div}(f_{1,\kappa})]$, 
o\`u $f_{1,\kappa}$ est une fonction r\'eguli\`ere dans $U$ non diviseur de $0$. 
On d\'esigne par $Z_{1,\kappa}$ le $\K$-sous-espace analytique ferm\'e (de dimension $n-1$) $f_{1,\kappa}^{-1}(\{0\})$, que l'on consid\`ere comme un 
$\K$-espace analytique de dimension $n-1$~; on note $\iota_{1,\kappa}$ le morphisme de 
$\K$-espaces analytiques correspondant \`a l'inclusion $Z_{1,\kappa}\subset U$. 
L'action du courant $[{\rm div}(f_1)]$ s'exprime sous la forme 
$$
\omega \in \mathscr A^{n-1,n-1}_c(U) \longmapsto 
\big\langle [{\rm div} (f_1)]\,,\omega\big\rangle := 
\sum_\kappa \int_{Z_{1,\kappa}\cap U} 
\omega = \sum_\kappa \int_{\iota_{1,\kappa}^{-1}(U)} \iota_{1,\kappa}^* (\omega).
$$
Pour chaque indice $\kappa$, on d\'efinit, pour tout $\lambda \in \R^*$, 
suivant le lemme 4.6.1 de \cite{ChLD}, un  $(1,1)$-courant $\big[|f_2|^\lambda/\lambda]\big]\, 
[{\rm div}(f_{1,\kappa})]$ par~: 
$$ 
\Big\langle 
\Big[\frac{|f_2|^\lambda}{\lambda}\Big] \, 
[{\rm div}(f_{1,\kappa})]\,,\omega \Big\rangle :=   
\int_{\iota_{1,\kappa}^{-1}(U)} 
\Big(\frac{\iota_{1,\kappa}^* (\theta)\, |\iota_{1,\kappa}^* (f_2)|^\lambda}{\lambda}\Big)\, 
\iota_{1,\kappa}^* (\omega) 
$$
pour tout $\omega \in \mathscr A^{n-1,n-1}_c(U)$ 
($\theta$ d\'esignant encore une fonction lisse de support inclus dans 
$U_{f_2}$ identiquement \'egale \`a $1$ au voisinage 
de $Z_{1,\kappa}\cap {\rm Supp}\, \omega$, que l'on peut encore construire gr\^ace au th\'eor\`eme de partitionnement de l'unit\'e,  
\cite[proposition 3.3.6]{ChLD}). On observe 
en utilisant la formule de Stokes dans $Z_{1,\kappa}$ 
et le lemme 4.6.1 de \cite{ChLD} que  
\begin{multline*}  
\Big\langle 
d'd''\, \Big(\Big[\frac{|f_2|^\lambda}{\lambda}\Big]\Big) \, 
[{\rm div}(f_{1,\kappa})]\,,\omega \Big\rangle  
= - \int_{\iota_{1,\kappa}^{-1}(U_{f_2})} 
\Big( 
d' \Big(\frac{|\iota_{1,\kappa}^* f_2|^\lambda}{\lambda}\Big)\Big) 
\wedge d''\big(\iota_{1,\kappa}^* (\omega)\big) \\ 
= - 
\int_{\iota_{1,\kappa}^{-1}(U_{f_2})} 
|\iota_{1,\kappa}^*(f_2)|^{\lambda} \, d' \log| \iota_{1,\kappa}^*(f_2)|\wedge 
d''\big(\iota_{1,\kappa}^* (\omega)\big) 
\end{multline*}
pour tout $\omega \in \mathscr A^{n-2,n-2}_c(U)$. 
La limite lorsque $\lambda$ tend vers $0$ dans $\R^*$ (au sens de la convergence faible des $(2,2)$-courants sur 
$U$) de $d'd''\big[|f_2|^\lambda/\lambda\big]\, 
\big[{\rm div}(f_{1,\kappa})\big]$ existe  
et d\'efinit un $(2,2)$-courant de support inclus dans $Z_1\cap Z_2$ que l'on conviendra de noter $[{\rm div}(f_{1,\kappa})] 
\wedge [{\rm div}(f_2)]$ 
(en respectant pour l'instant cet ordre). On observe d'ailleurs que 
\begin{equation}\label{prod2diviseurs}  
[{\rm div}(f_{1,\kappa})]\wedge 
[{\rm div}(f_{2})] 
= (\iota_{1,\kappa})_* \Big(d'd'' \big[\log |\iota_{1,\kappa}^* f_2|\big]\Big).    
\end{equation} 
On pose, en respectant pour l'instant l'ordre,  
$$
[{\rm div}(f_{1})] \wedge [{\rm div}(f_2)] := 
\sum_\kappa \pm [{\rm div}(f_{1,\kappa})] \wedge [{\rm div}(f_2)],  
$$ 
puisque $[{\rm div}(f_1)]:= 
\sum_\kappa \pm [{\rm div}(f_{1,\kappa})] $ (voir \cite[lemme 4.6.4]{ChLD} et commentaire qui suit la d\'emonstration du dit lemme).  

\begin{prop}\label{prop2fonctions} Soient $f_1$ et $f_2$ deux fonctions m\'eromorphes dans un ouvert $U$ de $X$,
telles que ${\rm codim}_U \big({\rm Supp}\, \big([{\rm div}(f_1)]\big) \cap {\rm Supp}\, \big([{\rm div}(f_2)]\big)\big)\geq 2$. Pour tout
$(\lambda_1,\lambda_2) \in (\R^*)^2$, on d\'efinit un courant $T_{\lambda_1,\lambda_2}^{f_1,f_2}$
appartenant \`a $\mathscr D_{2,2}(U)$ en posant
$$
\forall\, \omega \in
\mathscr{A}^{n-2,n-2}_c(U),\quad
\big\langle T_{\lambda_1,\lambda_2}^{f_1,f_2},\omega \big\rangle
= - \int_{U} d'\Big(\frac{|f_2|^{\lambda_2}}{\lambda_2}\Big)
\wedge d'd'' \Big(\frac{|f_1|^{\lambda_1}}{\lambda_1}\Big)\wedge d''\omega
$$
apr\`es avoir d\'ecoup\'e cette int\'egrale suivant un partitionnement de l'unit\'e 
$1=\sum_\iota \varphi_\iota$ subordonn\'ee au support de $d''\omega$ afin d'en assurer la convergence. 
Alors, on a, au sens des courants
\begin{equation}\label{produit}
\lim\limits_{\stackrel{(\lambda_1,\lambda_2) \rightarrow (0,0)}{\lambda_1 \not=0,\ \lambda_2
\not=0}} T_{\lambda_1,\lambda_2}^{f_1,f_2} = [{\rm div}(f_1)] \wedge [{\rm div}(f_2)], 
\end{equation}
o\`u le courant figurant au membre de droite a \'et\'e pr\'ec\'edemment d\'efini en termes des non-diviseurs de z\'ero $f_{1,\kappa}$ 
figurant dans $f_j$  ($j=1,2$).  
\end{prop}

\begin{proof} 
On note $Z_1$ et $Z_2$ les sous-espaces analytiques ferm\'es (au sens de Zariski) de $U$ d\'efinis comme les supports des courants 
$[{\rm div}(f_1)]$ et $[{\rm div}(f_2)]$. 
Soit $\omega \in \mathscr A^{n-2,n-2}_c(U)$. 
Du fait de l'hypoth\`ese 
${\rm codim}_U \big({\rm Supp}\, \big([{\rm div}(f_1)]\big) \cap {\rm Supp}\, \big([{\rm div}(f_2)]\big)\big)\geq 2$, il r\'esulte 
du lemme 3.2.5 de 
\cite{ChLD} et de la d\'efinition de la dimension locale $d_{\K}(x)$ ($x\in U)$ 
comme le minimum des dimensions $\K$-analytiques des domaines $\K$-affinoides qui contiennent $x$ (voir par exemple 
\cite[d\'efinition 1.16]{Duc07}), que le support de la 
$(n-2,n-1)$-forme diff\'erentielle 
$d''\omega$ ne rencontre pas le sous-ensemble de Zariski $Z_1\cap Z_2$.      
D'apr\`es le lemme de partitionnement de l'unit\'e \cite[proposition 3.3.6]{ChLD}, 
on peut introduire dans $U$ une partition de l'unit\'e $1= \sum_\iota \varphi_\iota$ (par des fonctions lisses \`a support compact), 
subordonn\'ee au recouvrement 
du compact ${\rm Supp}\, (d''\omega)$ de $U$ par les deux ouverts $U_{f_1}$ et $U_{f_2}$. 
On donne le sens suivant \`a l'expression 
\begin{equation}\label{expressionProp11} 
- \int_{U} d'\Big(\frac{|f_2|^{\lambda_2}}{\lambda_2}\Big)
\wedge d'd'' \Big(\frac{|f_1|^{\lambda_1}}{\lambda_1}\Big)\wedge \varphi_\iota\, d''\omega
\end{equation} 
dans les deux cas (de fait sym\'etriques) o\`u ${\rm Supp}\, \varphi_\iota 
\subset U_{f_1}$ et ${\rm Supp}\, \varphi_\iota \subset U_{f_2}$. 
\begin{itemize} 
\item 
Dans le premier cas, le sens que l'on donne \`a l'expression 
\eqref{expressionProp11} est le suivant~: on introduit suivant le 
lemme 4.6.1 de \cite{ChLD} le courant $\big[|f_2|^{\lambda_2}/\lambda_2\big]$ 
et le sens que l'on donne \`a \eqref{expressionProp11} est 
\begin{multline}\label{expressioncas1}  
- \Big\langle d'\Big[ 
\frac{|f_2|^{\lambda_2}}{\lambda_2}\Big]\,,\, 
d'd'' \Big(\varphi\, \frac{|f_1|^{\lambda_1}}{\lambda_1}\Big) 
\wedge \varphi_\iota\, d''\omega \Big\rangle \\
= - \Big\langle d'\Big[ 
\frac{|f_2|^{\lambda_2}}{\lambda_2}\Big]\,,\, 
d'\big(\varphi\, |f_1|^{\lambda_1}\, d''(\varphi\, \log |f_1|)\big) \wedge \varphi_\iota\, d''\omega 
\Big\rangle \\
= - \Big\langle d'\Big[ 
\frac{|f_2|^{\lambda_2}}{\lambda_2}\Big]\,,\, 
|f_1|^{\lambda_1} \Big(\varphi\, d'd''(\log |f_1|) 
+ \lambda_1 d'(\varphi \log |f_1|) \wedge d''(\varphi \log |f_1|) 
\wedge \varphi_\iota \, d''\omega \Big\rangle 
\\ = - \lambda_1\, \Big\langle d'\Big[ 
\frac{|f_2|^{\lambda_2}}{\lambda_2}\Big]\,,\, \varphi\, 
|f_1|^{\lambda_1}\, d'(\varphi \log |f_1|) \wedge d''(\varphi \log |f_1|)
\wedge \varphi_\iota \, d''\omega \Big\rangle    
\end{multline} 
o\`u $\varphi$ est une fonction lisse de support inclus dans $U_{f_1}$, identiquement 
\'egale \`a $1$ au voisinage du support de $\varphi_\iota\, d''\omega$. 
On a utilis\'e ici le fait que $d'd''\log |f_1| =0$ dans $U_{f_1}$, cons\'equence de la formule de Lelong-Poincar\'e 
\cite[th\'eor\`eme 4.6.5]{ChLD}.
\item 
Dans le second cas, le sens que l'on donne \`a l'expression 
\eqref{expressionProp11} est le suivant~: 
\begin{equation}\label{expressioncas2}
- \Big\langle T_{\lambda_1}^{f_1}\,,\, 
d'\Big(\psi\, \frac{|f_2|^{\lambda_2}}{\lambda_2}\Big) 
\wedge \varphi_\iota\, d''\omega \Big\rangle = 
-\Big\langle T_{\lambda_1}^{f_1}\,,\, \varphi_\iota\, |f_2|^{\lambda_2}\, 
d'(\psi \log |f_2|)\wedge d''\omega\Big\rangle,           
\end{equation} 
o\`u $\psi$ est une fonction lisse de support inclus dans $U_{f_2}$, identiquement 
\'egale \`a $1$ au voisinage du support de $\varphi_\iota\, d''\omega$. 
\end{itemize} 
On sait d'autre part que dans l'ouvert $U_{f_j}$ ($j=1,2)$, 
on a l'identit\'e suivante entre fonctions lisses~: 
$$
\forall\, \lambda_j\in \R^*,\quad 
\frac{|f_j|^{\lambda_j}}{\lambda_j} 
= \sum\limits_{k=0}^\infty \lambda_j^{k-1} 
\frac{(\log |f_j|)^k}{k!},  
$$
la convergence \'etant uniforme sur tout compact de $U_{f_2}$.  
Si l'on note $[(\log |f_j| )^k]$ le $(0,0)$-courant dans $U$ 
associ\'e \`a la fonction $(\log |f_j|)^k$ suivant le lemme 
4.6.1 de \cite{ChLD}, on a donc 
les identit\'es courantielles 
$$
\Big[\frac{|f_j|^{\lambda_j}}{\lambda_j}\Big] 
= \frac{[1]}{\lambda_j} + \sum\limits_{k=1}^\infty 
\frac{\lambda_j^{k-1}}{k!} \, \big[(\log |f_j|)^k], 
$$  
la convergence de la s\'erie figurant au membre de droite 
\'etant entendue ici au sens faible dans $\mathscr D_{0,0}(U)$.  
On en d\'eduit 
\begin{equation}\label{relationscourants} 
\begin{split} 
& T_{\lambda_1}^{f_1} = d'd'' \Big[
\frac{|f_1|^{\lambda_1}}{\lambda_1}\Big] 
= \sum\limits_{k=1}^\infty 
\frac{\lambda_1^{k-1}}{k!} \, d'd''\, \big[(\log |f_1|)^k\big]\\ 
& 
d'\, \Big[\frac{|f_2|^{\lambda_2}}{\lambda_2}\Big]  
= \sum\limits_{k=1}^\infty 
\frac{\lambda_2^{k-1}}{k!}\, d'\big[(\log |f_2|)^k\big].
\end{split}  
\end{equation}  
On observe que chacune des contributions 
$$
(\lambda_1,\lambda_2) \longmapsto 
- \int_{U} d'\Big(\frac{|f_2|^{\lambda_2}}{\lambda_2}\Big)
\wedge d'd'' \Big(\frac{|f_1|^{\lambda_1}}{\lambda_1}\Big)\wedge \varphi_\iota\, d''\omega
$$
admet une limite lorsque $(\lambda_1,\lambda_2)$ tend vers $0$. 
Dans le premier cas, on a 
\begin{multline}\label{limite-cas1} 
\lim\limits_{\stackrel{(\lambda_1,\lambda_2) 
\rightarrow (0,0)}{\lambda_1,\lambda_2 \in \R^*}}
\Big(- \int_{U} d'\Big(\frac{|f_2|^{\lambda_2}}{\lambda_2}\Big)
\wedge d'd'' \Big(\frac{|f_1|^{\lambda_1}}{\lambda_1}\Big)\wedge \varphi_\iota\, d''\omega\Big) = \\
\lim\limits_{\stackrel{(\lambda_1,\lambda_2) 
\rightarrow (0,0)}{\lambda_1,\lambda_2 \in \R^*}}
\Big( -\lambda_1\, 
\sum\limits_{k=1}^\infty 
\frac{\lambda_2^{k-1}}{k!} 
\Big\langle
d'\, \big[(\log |f_2|)^k\big]\,,\, 
\varphi\, 
|f_1|^{\lambda_1}\, d'(\varphi \log |f_1|) \wedge d''(\varphi \log |f_1|)
\wedge \varphi_\iota \, d''\omega\Big\rangle\Big) = 0 
\end{multline}  
d'apr\`es le th\'eor\`eme de convergence domin\'ee. Dans le second cas, on a pour les m\^emes raisons  
\begin{equation}\label{limite-cas2} 
\begin{split}
&\qquad\lim\limits_{\stackrel{(\lambda_1,\lambda_2) 
\rightarrow (0,0)}{\lambda_1,\lambda_2 \in \R^*}}
\Big(- \int_{U} d'\Big(\frac{|f_2|^{\lambda_2}}{\lambda_2}\Big)
\wedge d'd'' \Big(\frac{|f_1|^{\lambda_1}}{\lambda_1}\Big)\wedge \varphi_\iota\, d''\omega\Big) \\ 
&\qquad= \lim\limits_{\stackrel{(\lambda_1,\lambda_2) 
\rightarrow (0,0)}{\lambda_1,\lambda_2 \in \R^*}}
\Big( 
\sum\limits_{k=1}^\infty 
\frac{\lambda_1^{k-1}}{k!} 
\Big\langle 
d'd''\, \big[(\log |f_1|)^k\big]\,,\, 
\varphi_\iota\, |f_2|^{\lambda_2}\, 
d'(\psi \log |f_2|)\wedge d''\omega\Big\rangle\Big) \\ 
&\qquad= \lim\limits_{\stackrel{\lambda_2 \rightarrow 0}{\lambda_2 \in \R^*}} 
\Big\langle [{\rm div}(f_1)]\,,\, 
\varphi_\iota\, |f_2|^{\lambda_2}\, 
d'(\psi \log |f_2|)\wedge d''\omega\Big\rangle\\
&\qquad= \lim\limits_{\stackrel{\lambda_2 \rightarrow 0}{\lambda_2 \in \R^*}} 
\Big\langle [{\rm div}(f_1)]\,,\, 
\varphi_\iota\, d'\, \Big(\frac{\psi |f_2|^{\lambda_2}}{\lambda_2}\Big)  
\wedge d''\omega\Big\rangle \\
&\qquad= \lim\limits_{\stackrel{\lambda_2 \rightarrow 0}{\lambda_2 \in \R^*}} 
\Big( 
\Big\langle 
\Big[\frac{|f_2|^{\lambda_2}}{\lambda_2}\Big]\, [{\rm div}(f_1)]\,,\, 
\varphi_\iota \, d'd''\omega \Big\rangle + \Big\langle 
\Big[\frac{|f_2|^{\lambda_2}}{\lambda_2}\Big]\, [{\rm div}(f_1)]\,,\, 
d'(\varphi_\iota)\wedge d''\omega \Big\rangle\Big). 
\end{split}
\end{equation} 
Or, dans le premier cas, on a, puisque le support de 
$\varphi_\iota$ ne rencontre pas $Z_1$, que pour tout 
$\lambda_2$ dans $\R^*$ 
$$
\Big\langle 
\Big[\frac{|f_2|^{\lambda_2}}{\lambda_2}\Big]\, [{\rm div}(f_1)]\,,\, 
\varphi_\iota \, d'd''\omega \Big\rangle + \Big\langle 
\Big[\frac{|f_2|^{\lambda_2}}{\lambda_2}\Big]\, [{\rm div}(f_1)]\,,\, 
d'(\varphi_\iota)\wedge d''\omega \Big\rangle = 0.  
$$
En tenant donc compte de \eqref{limite-cas1} et de \eqref{limite-cas2}, on constate 
que la limite lorsque $(\lambda_1,\lambda_2)$ tend vers $(0,0)$ dans 
$(\R^*)^2$ de l'expression 
$$
- \int_{U} d'\Big(\frac{|f_2|^{\lambda_2}}{\lambda_2}\Big)
\wedge d'd'' \Big(\frac{|f_1|^{\lambda_1}}{\lambda_1}\Big)\wedge  d''\omega:= 
- \sum\limits_{\iota \in I} \int_{U} d'\Big(\frac{|f_2|^{\lambda_2}}{\lambda_2}\Big)
\wedge d'd'' \Big(\frac{|f_1|^{\lambda_1}}{\lambda_1}\Big)\wedge \varphi_\iota\, d''\omega
$$
est \'egale \`a la limite lorsque $\lambda$ tend vers $0$ de 
$$
\sum\limits_{\iota \in I} \Big(
\Big\langle 
\Big[\frac{|f_2|^{\lambda}}{\lambda}\Big]\, [{\rm div}(f_1)]\,,\, 
\varphi_\iota \, d'd''\omega \Big\rangle + \Big\langle 
\Big[\frac{|f_2|^{\lambda}}{\lambda}\Big]\, [{\rm div}(f_1)]\,,\, 
d'\varphi_\iota\wedge d''\omega \Big\rangle\Big), 
$$
c'est-\`a-dire, puisque $\sum_{\iota \in I} \varphi_\iota =1$ et que par cons\'equent 
$\sum_{\iota \in I} d'\varphi_\iota =0$, \`a l'action sur $\omega$ du courant 
$[{\rm div}(f_1)]\wedge [{\rm div}(f_2)]$ tel qu'il a \'et\'e d\'efini avant l'\'enonc\'e de la proposition 
\ref{prop2fonctions}.  
\end{proof} 

\begin{remark}\label{rem2fonctions} {\rm On v\'erifie que $T_{\lambda_1,\lambda_2}^{f_1,f_2} = T_{\lambda_2,\lambda_1}^{f_2,f_1}$ pour tout couple 
$(\lambda_1,\lambda_2)$ de $(\R^*)^2$ et toute paire de fonctions m\'eromorphes $(f_1,f_2)$. On commence par  
d\'efinir dans $U_{f_1}\cup U_{f_2}$ les deux courants $[|f_2|^{\lambda}/\lambda_2] \, T^{\lambda_1}_{f_1}$ et 
$[|f_1|^{\lambda_1}/\lambda_1]\, T^{\lambda_2}_{f_2}$ de la mani\`ere suivante 
(sym\'etrique). Par exemple  
\begin{multline*} 
\forall\, \omega \in \mathscr A^{n-1,n-1}_c(U),
\quad  
\Big\langle \Big[\frac{|f_2|^{\lambda_2}}{\lambda_2}\Big]\, 
T_{\lambda_1}^{f_1}\,,\, \omega\Big\rangle := \\ 
\begin{cases} \Big\langle 
\Big[\frac{|f_2|^{\lambda_2}}{\lambda_2}\Big]\,,\, 
d'd''\Big(\frac{|f_1|^{\lambda_1}}{\lambda_1}\Big)\wedge \omega \Big\rangle\ {\rm si}\ {\rm Supp}\, \omega \subset U_{f_1} \\ \\
\Big\langle T_{\lambda_1}^{f_1}\,,\, \frac{|f_2|^{\lambda_2}}{\lambda_2}\, \omega\Big\rangle \ {\rm si}\  {\rm Supp}(\omega) \subset U_{f_2} 
\end{cases}     
\end{multline*}  
(dans le premier cas, on utilise le lemme 4.6.1 de \cite{ChLD}). 
Les deux d\'efinitions alternatives propos\'ees se recollent ici dans $U_{f_1}\cap U_{f_2}$. On d\'efinit alors les 
deux courants suivants dans $U_{f_1}\cup U_{f_2}$~:  
$$ 
d'\Big[\frac{|f_2|^{\lambda_2}}{\lambda_2}\Big] 
\wedge T_{\lambda_1}^{f_1} := d'\Big( 
\frac{|f_2|^{\lambda_2}}{\lambda_2}\, T_{\lambda_1}^{f_1}\Big)\ ,\  
d'\Big[\frac{|f_1|^{\lambda_1}}{\lambda_1}\Big] 
\wedge T_{\lambda_2}^{f_2} := d'\Big( 
\frac{|f_1|^{\lambda_1}}{\lambda_1}\, T_{\lambda_2}^{f_2}\Big).  
$$
On observe que le courant $\mu_{\lambda_1,\lambda_2}$ d\'efini comme la diff\'erence 
de ces deux courants est un courant $d''$-ferm\'e dans $U_{f_1}\cup U_{f_2}$~: en effet, on a par exemple, 
si $\alpha \in \mathscr A^{n-2,n-2}_c(U_{f_1})$, 
\begin{multline}\label{calculsremarque}  
\Big\langle d'\Big[\frac{|f_2|^{\lambda_2}}{\lambda_2}\Big] 
\wedge T_{\lambda_1}^{f_1}\,,\, d''\alpha 
\Big\rangle = - \Big\langle 
\Big[\frac{|f_2|^{\lambda_2}}{\lambda_2}\Big]\,,\, 
d'd'' \Big(\frac{|f_1|^{\lambda_1}}{\lambda_1}\Big)\wedge d'd''\alpha \Big\rangle \\
= - \Big\langle 
\Big[\frac{|f_2|^{\lambda_2}}{\lambda_2}\Big]\,,\, d'\Big( 
d''\Big(\frac{|f_1|^{\lambda_1}}{\lambda_1}\Big)\wedge d'd''\alpha\Big)\Big\rangle =  
\Big\langle d'\Big[\frac{|f_2|^{\lambda_2}}{\lambda_2}\Big]\,,\, 
d''\Big(\frac{|f_1|^{\lambda_1}}{\lambda_1}\Big)\wedge d'd''\alpha\Big\rangle \\ 
= \Big\langle d'\Big[\frac{|f_2|^{\lambda_2}}{\lambda_2}\Big]\,,\, 
d''\Big(\frac{|f_1|^{\lambda_1}}{f_1}\, d'd''\alpha\Big) \Big\rangle  =  
\Big\langle d'd'' \Big[\frac{|f_2|^{\lambda_2}}{\lambda_2}\Big]\,,\, 
\frac{|f_1|^{\lambda_1}}{\lambda_1}\, d'd''\alpha \Big\rangle \\ 
=  \Big\langle d'\Big[\frac{|f_1|^{\lambda_1}}{\lambda_1}\Big] 
\wedge T_{\lambda_2}^{f_2}\,,\, d''\alpha 
\Big\rangle~;   
\end{multline} 
le calcul est sym\'etrique dans $U_{f_2}$. 
On rappelle que l'on a par d\'efinition 
\begin{multline*} 
\forall\, \lambda_1,\lambda_2 \in \R^*,\quad  \forall\, \omega \in \mathscr A_c^{n-2,n-2}(U),  
\quad \Big\langle T_{\lambda_1,\lambda_2}^{f_1,f_2} 
- T_{\lambda_2,\lambda_1}^{f_2,f_1}\,,\,\omega \Big\rangle = 
	\sum\limits_{\iota \in I} \big\langle \mu_{\lambda_1,\lambda_2}, 
	\varphi_\iota \, d''\omega 
	\big\rangle \\ 	
= - \sum\limits_{\iota \in I}
\big\langle \mu_{\lambda_1,\lambda_2}\,,\, 
d'' \varphi_\iota \wedge \omega \rangle = 
\Big\langle \mu_{\lambda_1,\lambda_2}\,,\, \Big(\sum_{\iota \in I} d''\varphi_\iota)\Big) \wedge \omega\Big\rangle = 0. 
\end{multline*} 
On conviendra de noter par la suite pour tous $\lambda_1,\lambda_2\in \R^*$~:
\begin{equation}\label{notation}
T_{\lambda_1,\lambda_2}^{f_1,f_2} = d'd''\Big[ \frac{|f_2|^{\lambda_2}}{\lambda_2}\Big]
\wedge d'd''\Big[ \frac{|f_1|^{\lambda_1}}{\lambda_1}\Big], 
\end{equation}
l'ordre des deux facteurs \'etant ici sans importance, ce qui en coh\'erent avec le fait 
qu'il s'agisse (formellement) d'un produit de $2$-courants. 
On a de plus 
$$
[{\rm div}(f_1)]\wedge [{\rm div}(f_2)]= 
[{\rm div}(f_2)] \wedge [{\rm div}(f_1)].
$$ 
}
\end{remark}
\noindent
On peut envisager le cas o\`u l'on a trois fonctions m\'eromorphes r\'eguli\`eres sans diviseurs de z\'ero dans un ouvert $U$ de $X$. 
On rappelle que l'action du courant $[{\rm div}(f_1)]\wedge [{\rm div}(f_2)]$ sur 
$\omega \in \mathscr A_c^{n-2,n-2}(U)$ est d\'efinie par 
\begin{multline*} 
\sum_{\kappa}\Big( - 
\int_{\iota_{1,\kappa}^{-1}(U_{f_2})\cap Z_{1,\kappa}} 
d'(\chi_\kappa\, \log| \iota_{1,\kappa}^*(f_2)|)\wedge d''\big(\iota_{1,\kappa}^* (\omega)\big) \Big) \\
= \sum_\kappa \Big(-\int_{\iota_{1,\kappa}^{-1}(U)\cap Z_{1,\kappa}}
d'(\chi_\kappa\, \log| \iota_{1,\kappa}^*(f_2)|)\wedge d''\big(\iota_{1,\kappa}^* (\omega)\big) \Big),     
\end{multline*} 
o\`u $\chi_\kappa$ est une fonction lisse sur le $\K$-espace analytique $Z_1$ 
identiquent \'egale \`a $1$ au voisinage du support de 
la $(n-2,n-1)$-forme lisse $d''\big(\iota_{1,\kappa}^* (\omega)$ 
et de support inclus dans le plus grand ouvert de $Z_{1,\kappa}\cap \iota_{1,\kappa}^{-1}(U)$ dans lequel 
la fonction m\'eromorphe r\'eguli\`ere $\iota_{1,\kappa}^*(f_2)$ ne s'annule pas. 
Pour $\lambda \in \R^*$, on utilise dans chaque ouvert 
$\iota_{1,\kappa}^{-1}(U)$ le lemme 4.6.1 de \cite{ChLD} pour justifier la d\'efinition du $(3,3)$-courant 
$$
\big([{\rm div}(f_1)]\wedge [{\rm div}(f_2)]\big) \wedge [{\rm div}(f_3)]
$$  
de la mani\`ere suivante. On rappelle d'abord (voir \eqref{prod2diviseurs}) que 
$$
[{\rm div}(f_1)] \wedge [{\rm div}(f_2)] := 
\sum_\kappa \pm (\iota_{1,\kappa})_* \Big( 
d'd'' \big[\log |\iota_{1,\kappa}^* (f_2)|\big]\Big) = 
\sum_\kappa \pm (\iota_{1,\kappa})_* \Big( 
[{\rm div}(\iota_{1,\kappa}^*(f_2)]\Big) 
$$
d'apr\`es la formule de Lelong-Poincar\'e appliqu\'ee \`a la fonction m\'eromorphe 
r\'eguli\`ere $\iota_{1,\kappa}^*(f_2)$ dans l'ouvert $\iota_{1,\kappa}^{-1}(U)$ du $\K$-espace analytique 
(de dimension $n-1$) $Z_{1,\kappa}$.  
On d\'efinit alors (en respectant pour l'instant l'ordre) 
$$ 
[{\rm div}(f_1)]\wedge [{\rm div}(f_2)] \wedge [{\rm div}(f_3)]
:= \sum_\kappa \pm\,  
(\iota_{1,\kappa})_*\Big( [{\rm div}(\iota_{1,\kappa}^* (f_2))] \wedge 
[{\rm div}(\iota_{1,\kappa}^* (f_3))]\Big). 
$$
Plus g\'en\'eralement, l'on aboutit \`a la d\'efinition suivante~: 
\begin{definition}{\rm 
Si $f_1,...,f_p$ sont $p$ fonctions r\'eguli\`erement m\'eromorphes dans $U$ telles que 
pour toute liste d'indices
$1\leq j_1 < \dots < j_k \leq p$ (avec $k=1,...,p$) on a
$$
{\rm codim}_{U}
\Big(\bigcap_{\ell =1}^k {\rm Supp}\, \big([{\rm div}(f_{j_\ell})]\big)\Big) \geq  k, 
$$ 
on d\'efinit inductivement pour tout $k$ entre $2$ et $p$  
\begin{equation}\label{defpproduit}  
[{\rm div}(f_1)]\wedge [{\rm div}(f_2] \wedge\cdots \wedge  [{\rm div}(f_{k})]
:= \sum_\kappa \pm\,  
(\iota_{1,\kappa})_*\Big( [{\rm div}(\iota_{1,\kappa}^* (f_2))] \wedge \cdots \wedge  
[{\rm div}(\iota_{1,\kappa}^* (f_k))]\Big). 
\end{equation}
}
\end{definition}
On se doit pour l'instant dans cette construction de respecter l'ordre dans lequel sont prises les fonctions 
r\'eguli\`erement m\'eromorphes $f_j$.  

\begin{theorem}\label{theorempfonctions}
Soient $f_1,...,f_p$ ($p\geq 1$) des fonctions m\'eromorphes r\'eguli\`eres dans un ouvert
$U$ d'un bon $\K$-espace analytique de Berkovich de dimension $n$ sans bord.
On fait l'hypoth\`ese que pour toute liste d'indices
$1\leq j_1 < \dots < j_k \leq p$ (avec $k=1,...,p$) on a
$$
{\rm codim}_{U} \Big(\bigcap_{\ell =1}^k {\rm Supp}\, \big([{\rm div}(f_{j_\ell})]\big)\Big) \geq  k.
$$
On d\'efinit un courant
$T_{\lambda_1,...,\lambda_p}^{f_1,...,f_p} \in \mathscr{D}_{p,p}(U,\R)$ en posant
$$
\forall\, \omega \in
\mathscr{A}^{n-p,n-p}_c(U),\quad
\langle T_{\lambda_1,...,\lambda_p}^{f_1,...,f_p},\omega \rangle
= - \int_{U} d'\Big(\frac{|f_p|^{\lambda_p}}{\lambda_p}\Big) 
\wedge \Big(\bigwedge\limits_{j=1}^{p-1} d'd'' \Big(\frac{|f_j|^{\lambda_j}}{\lambda_j}\Big)\Big)\wedge d''\omega
$$
apr\`es avoir d\'ecoup\'e cette int\'egrale suivant un partitionnement de l'unit\'e 
$1=\sum_\iota \varphi_\iota$ subordonn\'ee au support de $d''\omega$ afin de lui donner un sens et d'en assurer 
la convergence. 
Alors, ce courant ne d\'epend pas de l'ordre dans lequel sont prises les fonctions
$f_1,...,f_p$ et l'on a donc pour toute permutation $\sigma$ de $\{1,...,p\}$ 
\begin{equation}\label{produitgen}
\lim\limits_{\stackrel{(\lambda_1,...,\lambda_p) \rightarrow (0,...,0)}{\lambda_1 \not=0,...,  \lambda_p
\not=0}} T_{\lambda_1,...,\lambda_p}^{f_1,...,f_p} = [{\rm div}(f_1)] \wedge\dots \wedge  [{\rm div}(f_p)] = {[\rm div} (f_{\sigma(1)})] 
\wedge \dots \wedge [{\rm div}(f_{\sigma(p)})], 
\end{equation}
o\`u l'op\'eration multiplicative entre courants $[{\rm div}(f_j)]$ 
(respectant un ordre a priori impos\'e) a \'et\'e introduite pr\'ealablement. 
\end{theorem}

\begin{proof} Le r\'esultat est acquis pour $p=2$ d'apr\`es la proposition \ref{prop2fonctions}
et la remarque \ref{rem2fonctions}. On le suppose donc acquis (hypoth\`ese de r\'ecurrence) pour $p-1$ fonctions 
m\'eromorphes ($p\geq 3$). 
On note, pour $j=1,...,p$, $Z_j$  les sous-espaces analytiques ferm\'es (au sens de Zariski) de $U$ de codimension 
$1$ d\'efinis comme les supports 
des courants $[{\rm div}(f_j)]$. Pour chaque $j=1,...,p$, on note $\widehat{Z}_j$ l'intersection des sous-ensembles 
$\K$-analytiques $Z_\ell$ pour $\ell=1,...,j-1,j+1,...,p$.
Si $\widehat{Z_j}$ est non vide (${\rm codim}_{U}
\widehat{Z_j} < +\infty$), on a n\'ecessairement ${\rm codim}_{U} \widehat {Z_j} = p-1$ car cette codimension est 
minor\'ee par $p-1$ par hypoth\`eses et que $\widehat{Z_j}$ 
est d\'efini comme lieu des z\'eros communs d'exactement $p-1$ \'equations. On peut donc consid\'erer
$\widehat {Z_j}$ comme un $\K$-espace analytique de dimension $n-p+1$. Soit $\omega \in \mathscr A^{n-p,n-p}_c(U)$. 
Du fait de l'hypoth\`ese que pour toute liste d'indices
$1\leq j_1 < \dots < j_k \leq p$ (avec $k=1,...,p$)
$$
{\rm codim}_{U} \Big(\bigcap_{\ell =1}^k {\rm Supp}\, \big([{\rm div}(f_{j_\ell})]\big)\Big) \geq  k,
$$ il r\'esulte du lemme 3.2.5 de 
\cite{ChLD} et de la d\'efinition de la dimension locale $d_{\K}(x)$ ($x\in U)$ 
comme le minimum des dimensions $\K$-analytiques des domaines $\K$-affinoides qui contiennent $x$ (voir par exemple 
\cite{Duc07}, d\'efinition 1.16), que le support de la 
forme diff\'erentielle $d''\omega$ de bidegr\'e $(n-p,n-p+1)$ 
ne rencontre pas le sous-ensemble de Zariski $Z_1\cap Z_2\cdots\cap Z_p$.      
D'apr\`es le lemme de partitionnement de l'unit\'e \cite[proposition 3.3.6]{ChLD}, 
on peut introduire dans $U$ une partition de l'unit\'e $1= \sum_\iota \varphi_\iota$ (par des fonctions lisses 
\`a support compact), subordonn\'ee au recouvrement du 
compact ${\rm Supp}\, (d''\omega)$ de $U$ par les $p$ ouverts $U_{f_1}$, $U_{f_2}$,..., $U_{f_p}$. 
\noindent 
Dans un premier temps, \'etant donn\'es $p-1$ fonctions m\'eromorphiquement 
r\'eguli\`eres $g_1,...,g_{p-1}$ satisfaisant la condition 
$$
{\rm codim}_{U} \Big(\bigcap_{\ell =1}^k {\rm Supp}\, \big([{\rm div}(g_{j_\ell})]\big)\Big) \geq  k 
$$
pour tout $k$ entre $1$ et $p-1$ et tous $1\leq j_1 < \cdots <j_k\leq p-1$, 
il nous faut donner un sens (\`a l'aide du lemme 4.6.1 de 
\cite{ChLD}) au courant  
$$
\Big[\frac{|g_1|^{\lambda_1}}{\lambda_1}\Big] \, 
T_{\lambda_2,...,\lambda_{p-1}}^{g_2,...,g_{p-1}} \in \mathscr D_{p-2,p-2}(U). 
$$
On proc\`ede pour cela ainsi~: \'etant donn\'e
$\eta \in \mathscr A_c^{n-p+2,n-p+2}(U)$, on introduit, puisque le support 
de la forme $\eta$ ne rencontre pas $Z_{g_1} \cap \dots \cap Z_{g_{p-1}}$ d'apr\`es le lemme 3.2.5 de \cite{ChLD}, 
une partition de l'unit\'e $\sum_\iota \tau_\iota$ du 
support de 
$\eta$ subordonn\'ee au recouvrement de ce support par les ouverts 
$U_{g_\ell}$, $\ell=1,...,p-1$. On d\'efinit alors 
\begin{multline}\label{produitaux} 
\Big\langle \Big[\frac{|g_1|^{\lambda_1}}{\lambda_1}\Big] \, 
T_{\lambda_2,...,\lambda_{p-1}}^{g_2,...,g_{p-1}}\,,\, 
\tau_\iota\, \eta \Big\rangle \\
:= 
\begin{cases} 
\Big\langle T_{\lambda_2,...,\lambda_{p-1}}^{g_2,...,g_{p-1}}\,,\, 
\Big(\displaystyle{\frac{|g_1|^{\lambda_1}}{\lambda_1}}\Big)\, \tau_\iota \, \eta\Big\rangle 
\ {\rm si}\ {\rm Supp}(\tau_\iota)\subset U_{g_1} \\ \\
\Big\langle \Big[\displaystyle{\frac{|g_1|^{\lambda_1}}{\lambda_1}}\Big] \, 
T_{\lambda_2,...,\widehat{\lambda_{j_0}},\cdots \lambda_{p-1}}^{g_2,...,
\widehat{g_{j_0}},...,g_{p-1}}\,,\, d'd''\Big(\frac{|g_{j_0}|^{\lambda_{j_0}}}{\lambda_{j_0}}\Big)\wedge \tau_\iota\, 
\eta \Big\rangle\ 
{\rm si}\ {\rm Supp}(\tau_\iota) \subset U_{g_{j_0}}.
\end{cases}  
\end{multline} 
La seconde alternative se traite inductivement et la construction finit 
par se conclure lorsque $p=2$ \`a l'application du lemme 4.6.1 de 
\cite{ChLD}. La construction montre aussi (par r\'ecurrence sur le nombre 
de fonctions $g_j$ en jeu) que la limite 
\begin{equation}\label{limitepfonctions} 
\lim\limits_{\stackrel{(\lambda_1,...,\lambda_{p-1}) \rightarrow (0,...,0)}
{\lambda_1,...,\lambda_{p-1}\in \R^*}} 
\Big(\Big[\frac{|g_1|^{\lambda_1}}{\lambda_1}\Big] \, 
T_{\lambda_2,...,\lambda_{p-1}}^{g_2,...,g_{p-1}}\Big) 
\end{equation} 
existe inconditionnellement dans $\mathscr D_{p-2,p-2}(U)$ (au sens faible de la convergence des courants).  
\vskip 1mm
\noindent 
On est maintenant en mesure de donner le sens suivant \`a l'expression 
\begin{equation}\label{expressionpfonctions} 
- \int_{U} d'\Big(\frac{|f_p|^{\lambda_p}}{\lambda_p}\Big) 
\wedge \Big(\bigwedge\limits_{j=1}^{p-1} d'd'' \Big(\frac{|f_j|^{\lambda_j}}{\lambda_j}\Big)\Big)\wedge \varphi_\iota 
d''\omega
\end{equation} 
suivant que le support de $\varphi_\iota$ est inclus dans l'un des $U_{f_j}$ pour $j=1,...,p-1$ ou que le support de 
$\varphi_\iota$ est inclus dans $U_{f_p}$. 
\begin{itemize}
\item
Si ${\rm supp}\, \varphi_\iota \subset U_{f_{j_0}}$ pour un indice $j_0$ entre
$1$ et $p-1$, on d\'efinit l'expression \eqref{expressionpfonctions} par 
\begin{multline}\label{expressionpfonctionscas1}
- \int_{U} d'\Big(\frac{|f_p|^{\lambda_p}}{\lambda_p}\Big) 
\wedge \Big(\bigwedge\limits_{j=1}^{p-1} d'd'' \Big(\frac{|f_j|^{\lambda_j}}{\lambda_j}\Big)\Big)\wedge \varphi_\iota 
d''\omega\\ 
:= 
- \Big\langle d'
\Big( \Big[\frac{|f_p|^{\lambda_p}}{\lambda_p}\Big]\, 
T_{\lambda_1,...,\widehat{\lambda_{j_0}},...,\lambda_{p-1}}^{f_1,...,\widehat{f_{j_0}},...,f_{p-1}}\Big)\,,\, d'd''\Big(
\frac{\varphi |f_{j_0}|^{\lambda_{j_0}}}{\lambda_{j_0}}\Big) \wedge \varphi_\iota d''\omega\Big\rangle \\ 
= - \lambda_{j_0} 
\, \Big\langle 
d'
\Big( \Big[\frac{|f_p|^{\lambda_p}}{\lambda_p}\Big]\, 
T_{\lambda_1,...,\widehat{\lambda_{j_0}},...,\lambda_{p-1}}^{f_1,...,\widehat{f_{j_0}},...,f_{p-1}}\Big)\,,\, 
\varphi\, 
|f_{j_0}|^{\lambda_{j_0}}\, d'(\varphi \log |f_{j_0}|) \wedge d''(\varphi \log |f_{j_0}|)
\wedge \varphi_\iota \, d''\omega \Big\rangle,  
\end{multline} 
o\`u $\varphi$ d\'esigne une fonction lisse identiquement \'egale \`a $1$ au voisinage du support de $\varphi_\iota \, 
d''\omega$ et de support inclus dans $U_{f_{j_0}}$. 
\item Si ${\rm Supp}\, \varphi_\iota \subset U_{f_p}$, 
on d\'efinit l'expression \eqref{expressionpfonctions} par 
\begin{multline}\label{expressionpfonctionscas2}
- \int_{U} d'\Big(\frac{|f_p|^{\lambda_p}}{\lambda_p}\Big) 
\wedge \Big(\bigwedge\limits_{j=1}^{p-1} d'd'' \Big(\frac{|f_j|^{\lambda_j}}{\lambda_j}\Big)\Big)\wedge \varphi_\iota 
d''\omega\\ 
:= 
- \Big\langle
T_{\lambda_1,...,\lambda_{p-1}}^{f_1,...,f_{p-1}}\,,\, 
d'\Big( \varphi \frac{|f_p|^{\lambda_p}}{\lambda_p}\Big)\wedge \varphi_\iota\, d''\omega\Big\rangle. 
\end{multline}  
\end{itemize} 
On \'etudie maintenant le comportement de chaque fonction 
\begin{equation}\label{fonctionpfonctions} 
(\lambda_1,...,\lambda_p) \longmapsto 
\int_{U} d'\Big(\frac{|f_p|^{\lambda_p}}{\lambda_p}\Big) 
\wedge \Big(\bigwedge\limits_{j=1}^{p-1} d'd'' \Big(\frac{|f_j|^{\lambda_j}}{\lambda_j}\Big)\Big)\wedge \varphi_\iota 
d''\omega
\end{equation} 
lorsque $(\lambda_1,...,\lambda_p)$ tend vers $(0,...,0)$ inconditionnellement dans $(\R^*)^p$ suivant que l'on se 
trouve dans l'un des deux cas distingu\'es ci-dessus. 
\begin{itemize} 
\item Si ${\rm Supp}(\varphi_\iota)\subset U_{f_{j_0}}$ pour 
$j_0\in \{1,...,p-1\}$, on peut remplacer dans le crochet figurant au membre de droite  
\eqref{expressionpfonctionscas1} l'expression $|f_{j_0}|^{\lambda_{j_0}}$ par 
$$
|f_{j_0}|^{\lambda_{j_0}} = \sum\limits_{k=0}^\infty 
\frac{\lambda_{j_0}^k}{k!} \, (\log |f_{j_0}|)^k. 
$$
Il r\'esulte alors du fait que la limite inconditionnelle \eqref{limitepfonctions} existe au sens faible dans 
$\mathscr D_{p-1,p-1}(U)$ que la fonction 
\eqref{fonctionpfonctions} tend vers $0$ lorsque 
$(\lambda_1,...,\lambda_p)$ tend inconditionnellement vers 
$(0,...,0)$ dans $(\R^*)^p$.
\item Si ${\rm Supp}(\varphi_\iota) \subset U_{f_1}$, on observe 
que l'expression \eqref{expressionpfonctionscas2} s'exprime aussi 
\begin{multline*} 
- \int_{U} d'\Big(\frac{|f_p|^{\lambda_p}}{\lambda_p}\Big) 
\wedge \Big(\bigwedge\limits_{j=1}^{p-1} d'd'' \Big(\frac{|f_j|^{\lambda_j}}{\lambda_j}\Big)\Big)\wedge \varphi_\iota 
d''\omega\\ 
= 
- \sum\limits_{k=1}^\infty 
\frac{\lambda_p^{k-1}}{k!} 
\Big\langle
T_{\lambda_1,...,\lambda_{p-1}}^{f_1,...,f_{p-1}}\,,\, 
(\varphi \log |f_p|)^{k}\, 
d'(\varphi \log|f_p|) \wedge \varphi_\iota\, d''\omega\Big\rangle. 
\end{multline*} 
et la fonction \eqref{fonctionpfonctions} admet d'apr\`es l'hypoth\`ese de r\'ecurrence comme limite inconditionnelle 
lorsque $(\lambda_1,...,\lambda_p)$ tend vers 
$(0,...,0)$ dans $(\R^*)^p$ l'expression 
\begin{multline}\label{limitepfonctionscas2}  
\Big\langle 
[{\rm div}(f_1)]\wedge \cdots \wedge [{\rm div}(f_{p-1})]\,,\, 
\log |f_p| \, \varphi_\iota\, d'd''\omega\Big\rangle 
\\
+ \Big\langle [{\rm div}(f_1)]\wedge \cdots \wedge [{\rm div}(f_{p-1})]\,,\, 
\log|f_p|\, d'\varphi_\iota \wedge d''\omega\Big\rangle \\ 
= \Big\langle \big[\log |f_p|\big]\, T^{f_1,...,f_{p-1}}_{0,...,0}\,,\, 
\varphi_\iota\, d'd''\omega\Big\rangle 
\\
+ \Big\langle \big[\log |f_p|\big]\, 
T^{f_1,...,f_{p-1}}_{0,...,0}\,,\, d'\varphi_\iota \wedge d''\omega\Big\rangle, 
\end{multline} 
o\`u le courant 
$$
\big[\log |f_p|\big]\, T^{f_1,...,f_{p-1}}_{0,...,0} 
= [\log |f_p|]\, \Big([{\rm div}(f_1)]\wedge \cdots \wedge [{\rm div}(f_{p-1})]\Big) 
$$
est bien d\'efini gr\^ace au lemme 4.6.1 de \cite{ChLD} compte-tenu de l'expression inductive du courant 
$[{\rm div}(f_1)] \wedge \cdots \wedge [{\rm div}(f_{p-1})]$ 
(de support par construction m\^eme inclus dans $Z_1 \cap \dots \cap Z_{p-1}$).   
\end{itemize} 
On remarque \'egalement que si ${\rm Supp}\, (\varphi_\iota)\subset U_{j_0}$ avec 
$1\leq j_0\leq p-1$, alors 
\begin{equation}\label{limitepfonctionscas2bis} 
\Big\langle \big[\log |f_p|\big]\, T^{f_1,...,f_{p-1}}_{0,...,0}\,,\, 
\varphi_\iota\, d'd''\omega\Big\rangle 
+ \Big\langle \big[\log |f_p|\big]\, 
T^{f_1,...,f_{p-1}}_{0,...,0}\,,\, d'\varphi_\iota \wedge d''\omega\Big\rangle = 0 
\end{equation} 
compte-tenu du fait que le support du courant 
$$
\big[\log |f_p|\big]\, T^{f_1,...,f_{p-1}}_{0,...,0} 
= [\log |f_p|]\, \Big([{\rm div}(f_1)]\wedge \cdots \wedge [{\rm div}(f_{p-1})]\Big)
$$
est inclus dans le sous-ensemble de Zariski $Z_1\cap \cdots \cap Z_{p-1}$ de codimension $p-1$. 
On en d\'eduit donc que la limite lorsque $(\lambda_1,...,\lambda_p)$ tend vers $(0,...,0)$ incondionnellement dans 
$(\R^{*})^p$ de 
$$
\Big\langle T^{f_1,...,f_p}_{\lambda_1,...,\lambda_p}\,,\, 
\omega \Big\rangle := - 
\sum\limits_{\iota \in I} 
\int_{U} d'\Big(\frac{|f_p|^{\lambda_p}}{\lambda_p}\Big) 
\wedge \Big(\bigwedge\limits_{j=1}^{p-1} d'd'' \Big(\frac{|f_j|^{\lambda_j}}{\lambda_j}\Big)\Big)\wedge \varphi_\iota 
d''\omega
$$
 existe et vaut 
\begin{multline*} 
\sum\limits_{\iota \in I} \Big(\Big\langle 
\big[\log |f_p|\big]\, [{\rm div}(f_1)]\wedge \cdots \wedge [{\rm div}(f_{p-1})]\,,\, 
\varphi_\iota\, d'd''\omega\Big\rangle 
\\
+ \Big\langle \big[\log |f_p|\big]\, 
[{\rm div}(f_1)]\wedge \cdots \wedge [{\rm div}(f_{p-1})]\,,\, 
d'\varphi_\iota \wedge d''\omega\Big\rangle\Big) \\ 
=\Big\langle 
\big[ \log |f_p|\big]\, 
[{\rm div}(f_1)]\wedge \cdots \wedge [{\rm div}(f_{p-1})]\,,\, 
d'd''\omega\Big\rangle = 
\Big\langle 
[{\rm div}(f_1)]\wedge \dots \wedge [{\rm div}(f_p)]\,,\, \omega\Big\rangle 
\end{multline*} 
compte-tenu de la d\'efinition inductive des courants  
$[{\rm div}(f_1)\wedge \cdots \wedge [{\rm div}(f_k)]$ pour 
$k=2,...,p$ et de la formule de Lelong-Poincar\'e sur le ferm\'e de Zariski 
$Z_1\cap \cdots \cap Z_{p-1}$ (de codimension $p-1$) consid\'er\'e comme un 
$\K$-espace analytique de dimension $n-(p-1)$. 
\hfil\break 
Il reste \`a justifier l'\'egalit\'e 
$$
T_{\lambda_{\tau(1)},...,\lambda_{\tau(p)}}^{f_{\tau(1)},...,f_{\tau(p)}}
= T_{\lambda_1,...,\lambda_p}^{f_1,...,f_p} 
$$
pour toute transposition $\tau$ de $\mathscr S_{\{1,...,p\}}$ (groupe des permutations). 
Le r\'esultat est acquis du fait de l'hypoth\`ese de r\'ecurrence lorsque $\tau(p)=p$ et 
l'on peut se ramener ainsi \`a prouver 
l'\'egalit\'e courantielle 
\begin{equation}\label{egalitecourantielle} 
T_{\lambda_1,...,\lambda_{p-2},\lambda_{p-1},\lambda_p}^{f_1,...,f_{p-2},f_{p-1},f_p} = 
T_{\lambda_1,...,\lambda_{p-2},\lambda_{p},\lambda_{p-1}}^{f_1,...,f_{p-2},f_{p},f_{p-1}}. 
\end{equation} 
On d\'efinit pour cela comme dans la remarque \ref{rem2fonctions} le courant 
$$
\mu_{\lambda_1,...,\lambda_p} 
= d'\Big( 
\Big[\frac{|f_p|^{\lambda_p}}{\lambda_p}\Big] 
\, T_{\lambda_1,...,\lambda_{p-2},\lambda_{p-1}}^{f_1,...,f_{p-2},f_{p-1}}\Big) 
- d'\Big( 
\Big[\frac{|f_{p-1}|^{\lambda_{p-1}}}{\lambda_{p-1}}\Big]\, 
T_{\lambda_1,...,\lambda_{p-2},\lambda_p}^{f_1,...,f_{p-2},f_p}\Big)\,   
$$
o\`u la multiplication des courants s'effectue suivant la d\'emarche 
inductive \eqref{produitaux}. 
Si l'on remarque que 
\begin{eqnarray*} 
T_{\lambda_1,...,\lambda_{p-2},\lambda_{p-1}}^{f_1,...,f_{p-2},f_{p-1}} = 
d'd''\Big( 
\Big[\frac{|f_{p-1}|^{\lambda_{p-1}}}{\lambda_{p-1}}\Big] 
\, T_{\lambda_1,...,\lambda_{p-2}}^{f_1,...,f_{p-2}}\Big) \ ,\ 
T_{\lambda_1,...,\lambda_{p-2},\lambda_{p}}^{f_1,...,f_{p-2},f_{p}} = 
d'd''\Big( 
\Big[\frac{|f_{p}|^{\lambda_{p}}}{\lambda_p}\Big] 
\, T_{\lambda_1,...,\lambda_{p-2}}^{f_1,...,f_{p-2}}\Big), 
\end{eqnarray*} 
on se met en situation de reprendre les calculs \eqref{calculsremarque}.  
On v\'erifie 
que le courant $\mu_{\lambda_1,...,\lambda_p}$ est $d''$-ferm\'e dans l'union des ouverts 
$U_{f_j}$ pour $j=1,...,p$~: 
\begin{itemize} 
\item 
si $\alpha\in \mathscr A_c^{n-p,n-p}(U)$ est de support 
dans $U_{f_{p-1}}\cup U_{f_p}$, les calculs 
sont identiques \`a ceux conduits dans \eqref{calculsremarque} et le courant $d'$ et $d''$-ferm\'e 
$T_{\lambda_1,...,\lambda_{p-2}}^{f_1,...,f_{p-2}}$ qui y 
intervient y joue un r\^ole neutre~; 
\item si d'autre part $\alpha\in \mathscr A_{c}^{n-p,n-p}(U)$ est de support 
dans $U_{f_{j_0}}$ pour $j_0$ entre $1$ et $p-2$, on est amen\'e \`a remplacer  
la forme $\alpha$ par la forme lisse $\alpha \wedge d'd''(|f_{j_0}|^{\lambda_{j_0}}/\lambda_{j_0})$ 
et \`a \'eliminer ainsi la fonction $f_{j_0}$ de la liste 
$[f_1,...,f_{p-2}]$, ce qui permet d'abaisser le nombre de 
fonctions $f_1,...,f_{p-2}$. 
\end{itemize}
Comme dans la remarque \ref{rem2fonctions}, on conclut \`a l'\'egalit\'e courantielle 
\eqref{egalitecourantielle}.  Ce qui ach\`eve la preuve du th\'eor\`eme \ref{theorempfonctions}.   
\end{proof}
\section{R\'ealisation \`a la Mellin de courants de Green normalis\'es}\label{sectionGreen}

Dans cette section, comme dans la pr\'ec\'edente, $X$ d\'esigne un bon espace de Berkovich sans bord de dimension pure $n$.
\vskip 1mm
\noindent
Soit $\mathscr{L} \rightarrow U$ un fibr\'e en droites au-dessus d'un ouvert $U$ de $X$ \'equip\'e d'une 
m\'etrique continue $\|\ \|=\exp(-\rho)$, o\`u $\rho$ est une fonction continue r\'eelle. On pourra se repporter 
\`a \cite[section 6.2]{ChLD} pour la notion de fibr\'e en droites 
avec une m\'etrique et \`a \cite[section 6.4.1]{ChLD} pour la d\'efinition de la forme de Chern ou du courant de Chern suivant que la m\'etrique $\|\ \|$ est lisse ou non.
\'Etant donn\'ee une section m\'eromorphe $s$ du fibr\'e $\mathscr{L}$ au-dessus de l'ouvert $U$, on convient d'appeler courant de Green normalis\'e subordonn\'e au courant 
$[{\rm div}(s)]$ dans $U$ un courant $G\in\mathscr D_{0,0}(U)$ tel que 
$$
d'd'' G + [{\rm div}(s)] = c_1(\mathscr{L},\|\ \|), 
$$
o\`u $\|\ \|$ d\'esigne une m\'etrique continue sur le fibr\'e en droites 
$\mathscr{L}$ et $c_1(\mathscr{L},\|\ \|)$ d\'esigne le $(1,1)$-courant 
de Chern associ\'e \`a la m\'etrique $\|\ \|$.  
Lorsque cette m\'etrique est lisse, il en est de m\^eme de la premi\`ere forme de Chern que l'on convient de noter pour simplifier 
$c_1(\mathscr {L},\|\ \|)$ et le $(0,0)$-courant $G$ est alors un courant de Green pour $[{\rm div}(s)]$, au sens 
o\`u $d'd''G + [{\rm div}(s)]$ est un $(1,1)$-courant de la forme $\varphi \mapsto \int_U \omega \wedge \varphi$, o\`u 
$\omega = c_1(\mathscr {L},\|\ \|)$ est une forme lisse.
\vskip 1mm
\noindent
Soit $\omega\in \mathscr{A}_c^{n,n}(U)$. Le support (compact) de $\omega$ \'evite tout sous-ensemble ferm\'e de Zariski d'int\'erieur 
vide \cite[lemme 3.2.5]{ChLD} et l'on peut 
donc affirmer qu'il existe, pour tout $x\in {\rm Supp}(\omega)$,  un voisinage $V_x$
de $x$ dans $U$ au-dessus duquel le fibr\'e $\mathscr{L}$ admet un rep\`ere
$\sigma_{V_x}$ dans lequel la section $s$ s'exprime sous la forme $f_{V_x}
\sigma_{V_x}$, o\`u $f_{V_x}$ est une fonction r\'eguli\`ere inversible
dans $V_x$. 
\begin{definition}{\rm 
Soit $s~: U\to \mathscr L$ une section m\'eromorphe du fibr\'e $\mathscr L$ au-dessus de $U$, \'equip\'e d'une m\'etrique lisse $\|\ \|$.  
On d\'efinit donc, pour tout $\lambda \in \R^*$, un \'el\'ement de $\mathscr{D}_{0,0}(U)$ par~:
$$
G_\lambda^{s} = -\Big[\frac{\|s\|^{\lambda}}{\lambda}\Big]~:
\omega \in \mathscr{A}^{n,n}_c(U) \longmapsto
-\int_{U} \frac{\|s\|^{\lambda}}{\lambda}\, \omega.
$$
}
\end{definition}
\noindent
Il r\'esulte de la formule de Lelong-Poincar\'e que l'on a, au sens des courants dans $U$,
\begin{equation}\label{poinclel2}
\lim\limits_{\stackrel{\lambda \rightarrow 0}{\lambda \not=0}}
(d'd'' G_{\lambda}^{s}) + [{\rm div}(s)] = c_1(\mathscr{L},\|\ \|).
\end{equation}
En effet, l'on a d'apr\`es la formule de Stokes ($X$ est suppos\'e sans bord), si $\varphi$ 
d\'esigne une fonction lisse identiquement 
\'egale \`a $1$ au voisinage du support de $d''\omega$ et de support compact dans $U$ 
(que l'on peut encore construire gr\^ace au th\'eor\`eme de partitionnement de l'unit\'e, \cite[proposition 3.3.6]{ChLD}), 
\begin{multline*}
\forall\,\omega\in\mathscr A_c^{n-1,n-1}(U),\quad \big\langle d'd'' G_\lambda^s\,,\,\omega\big\rangle=- 
\big\langle d' G_\lambda^s\,,\,d''\omega\big\rangle=-\int_{X^{\rm an}}d'\Big(\varphi\frac{\|s\|^\lambda}{\lambda}\Big)\wedge d''\omega\\
=-\int_{U_s}\|s\|^\lambda\,d'\big(\log\|s\|\big)\wedge d''\omega, 
\end{multline*}
avec $U_s~:= U\setminus Z$, o\`u $Z$ est le sous-espace analytique ferm\'e (au sens de Zariski) de $U$ d\'efinit comme le support du courant 
$[{\rm div}(s)]$.
\vskip 1mm
\noindent
D'apr\`es le th\'eor\`eme de convergence domin\'ee de Lebesgue, la fonction 
$\lambda \mapsto d'd''G_\lambda^s$ (\`a valeurs dans $\mathscr D_{1,1}(U)$) 
admet comme limite lorsque $\lambda$ tend vers $0$
\begin{multline}\label{eqgreen1}
 \forall\, \omega\in\mathscr A_c^{n-1,n-1}(U),\quad \longmapsto -\int_{U_s}d'\big(\log\|s\|\big)\wedge d''\omega
 =-\int_{U_s}d'd''\big(\log\|s\|\big)\wedge\omega\\=-\big\langle[{\rm div}(s)]-
 d'd''[\rho]\,,\, \omega\big\rangle.
\end{multline}
Ainsi on conclut de \eqref{eqgreen1}, avec  $c_1(\mathscr{L},\|\ \|)=d'd''[\rho]$, ce qui ach\`eve la justification de l'\'egalit\'e 
\eqref{poinclel2}.
\vskip 1mm
\noindent
Supposons maintenant que $\mathscr{L}_1 \rightarrow U$ et $\mathscr{L}_2\rightarrow U$ sont deux fibr\'es en droites au-dessus de $U$, 
chacun \'equip\'e d'une m\'etrique 
lisse $(e^{-\rho_{j,\iota}})_\iota$ (subordonn\'ee \`a un recouvrement $(V_\iota)_\iota$
de $U$ suffisamment fin pour que les deux fibr\'es se trivialisent au-dessus de chaque $V_\iota$), ce qui signifie que, pour chaque 
$\iota$, les deux fonctions 
$\rho_{j,\iota}$ s'expriment localement au voisinage $\xi$ de chaque point de $V_\iota$ comme des fonctions $C^\infty$
\`a valeurs r\'eelles de fonctions du type $\log|f_{\iota,\xi}|$ o\`u $f_{\iota,\xi}$ est une fonction
r\'eguli\`ere inversible. Pour $j=1,2$, les premiers courants de Chern
$c_1(\mathscr{L}_j,\|\ \|_j)$ sont dans ce cas associ\'es \`a des \'el\'ements de
$\mathscr{A}^{1,1}(U)$ (que l'on notera de la m\^eme mani\`ere, mais ce sont cette fois des $(1,1)$-formes diff\'erentielles dans $U$, 
que l'on traitera comme telles), 
dites premi\`eres formes de Chern
des fibr\'es $\mathscr{L}_j$ (chacun \'equip\'e de la m\'etrique lisse $\|\ \|_j$).
\vskip 1mm
\noindent
La proposition suivante s'inscrit dans la droite ligne de la proposition \ref{prop2fonctions}.

\begin{prop}\label{prop2fonctions2} Soient $\mathscr{L}_1 \rightarrow U$
et $\mathscr{L}_2 \rightarrow U$ deux fibr\'es en droites au-dessus d'un ouvert
$U$ d'un bon $\K$-espace de Berkovich $X$ sans bord, chacun \'equip\'e
d'une m\'etrique lisse $\|\ \|_j$. Soient $s_1$ et $s_2$ deux sections m\'eromorphes
respectivement de $\mathscr{L}_1$ et $\mathscr{L}_2$ telles que
$${\rm codim}_{U}({\rm Supp}\big([{\rm div}(s_1)]) \cap {\rm Supp}([{\rm div}(s_2)])\big)\geq 2.$$ Pour tout
$(\lambda_1,\lambda_2)\in (\R^*)^2$, on d\'efinit un \'el\'ement $G_{\lambda_1,\lambda_2}^{s_1,s_2}$ de $\mathscr{D}_{1,1}(U)$ par
$$
G_{\lambda_1,\lambda_2}^{s_1,s_2}~:
\omega \in \mathscr{A}^{n-1,n-1}_c(U)
\longmapsto - \int_{U} \frac{\|s_2\|_2^{\lambda_2}}{\lambda_2}\,
d'd'' \Big(\frac{\|s_1\|^{\lambda_1}_1}{\lambda_1}\Big) \wedge \omega
$$
apr\`es avoir d\'ecoup\'e cette int\'egrale suivant un partitionnement de l'unit\'e 
$1=\sum_\iota \varphi_\iota$ subordonn\'ee au support de $d''\omega$ afin d'en assurer la convergence. 
De plus on a, au sens de la convergence faible des courants sur $U$,
\begin{equation}\label{green1}
\begin{split}
& \lim\limits_{\stackrel{(\lambda_1,\lambda_2)
\rightarrow (0,0)}{\lambda_1\not=0,\lambda_2 \not=0}}
\Big(d'd'' \Big( G_{\lambda_1,\lambda_2}^{s_1,s_2} +
c_1(\mathscr{L}_1,\|\ \|_1) \wedge G_{\lambda_2}^{s_2}
+ c_1(\mathscr{L}_2,\|\ \|_2)
\wedge G_{\lambda_1}^{s_1}\Big)\Big) \\
& \qquad + [{\rm div}(s_1)] \wedge [{\rm div}(s_2)] = \big[c_1(\mathscr{L}_1,\|\ \|_1) \wedge c_1(\mathscr{L}_2,\|\ \|_2)\big], 
\end{split}
\end{equation}
o\`u le produit de courants $[{\rm div}(s_1)] \wedge [{\rm div}(s_2)]$ est  
d\'efini localement comme l'est le courant $[{\rm div}(f_1)]\wedge [{\rm div}(f_2)]$ 
dans la Proposition \ref{prop2fonctions} \`a partir des fonctions m\'eromorphes coordonn\'ees $f_1$ et $f_2$ respectivement de $s_1$ et $s_2$ dans les rep\`eres locaux pour les fibr\'es $\mathscr{L}_1$ et $\mathscr {L}_2$.   
\end{prop}

\begin{proof}
La preuve est similaire \`a celle de la proposition \ref{prop2fonctions}.
On note encore $Z_1$ et $ Z_2$ les sous-espaces analytiques ferm\'es (au sens de Zariski) de $U$ d\'efinis 
comme les supports des courants 
$[{\rm div}(s_1)]$ et $[{\rm div}(s_2)]$ et $U_{s_j} :=U \setminus Z_j$ ($j=1,2$). Notons $\iota_j~:  Z_j \rightarrow U$ 
les morphismes de $\K$-espaces 
analytiques correspondant aux inclusions
$ Z_j\subset U$ (o\`u $j=1,2$).
Soit $\omega\in \mathscr{A}^{n-1,n-1}_c(U)$. Du fait de l'hypoth\`ese 
${\rm codim}_U \big({\rm Supp}\, \big([{\rm div}(s_1)]\big) \cap {\rm Supp}\, \big([{\rm div}(s_2)]\big)\big)\geq 2$, 
il r\'esulte du lemme 3.2.5 de \cite{ChLD} et de la d\'efinition de la dimension locale $d_{\K}(x)$ ($x\in U)$ 
comme le minimum des dimensions $\K$-analytiques des domaines $\K$-affinoides qui contiennent $x$ (voir par exemple  
\cite[d\'efinition 1.16]{Duc07}), que le support de la  $(n-2,n-1)$-forme diff\'erentielle 
$d''\omega$ ne rencontre pas le sous-ensemble de Zariski $ Z_1\cap  Z_2$.      
D'apr\`es le lemme de partitionnement de l'unit\'e \cite[proposition 3.3.6]{ChLD}, 
on peut introduire dans $ U$ une partition de l'unit\'e $1= \sum_\iota \varphi_\iota$ (par des fonctions lisses \`a support compact), 
subordonn\'ee au recouvrement du 
compact ${\rm Supp}\, (d''\omega)$ de $ U$ par les deux ouverts $U_{s_1}$ et $ U_{s_2}$. Pour chaque indice $\iota$, l'int\'egrale
$$
- \int_{U} \frac{\|s_2\|_2^{\lambda_2}}{\lambda_2}
\, d'd'' \Big(\frac{\|s_1\|_1^{\lambda_1}}{\lambda_1}\Big)\wedge \varphi_\iota\, \omega
$$
est bien d\'efinie. Il est donc clair que l'on d\'efinit l'action d'un courant de
bidimension $(n-1,n-1)$ en posant
\begin{equation}\label{green1aux}
\big\langle G_{\lambda_1,\lambda_2}^{s_1,s_2},\omega \big\rangle
:= - \sum\limits_\iota \int_{U} \frac{\|s_2\|_2^{\lambda_2}}{\lambda_2}
\, d'd'' \Big(\frac{\|s_1\|_1^{\lambda_1}}{\lambda_1}\Big)\wedge \varphi_\iota\, \omega.
\end{equation}
Il r\'esulte de \eqref{poinclel2} et \eqref{eqgreen1} que  l'on a respectivement dans $U_{s_2}$ et $U_{s_1}$ les \'egalit\'es suivantes :
\begin{multline}\label{eqgreen2}
 \lim\limits_{\stackrel{(\lambda_1,\lambda_2)
\rightarrow (0,0)}{\lambda_1\not=0,\lambda_2 \not=0}} d'd''\Big(c_1(\mathscr L_1,\|\ \|_1)\wedge G_{\lambda_2}^{s_2}\Big)= 
c_1(\mathscr L_1,\|\ \|_1)\wedge\big( -[{\rm div}(s_2)] + c_1(\mathscr{L}_2,\|\ \|_2)\big)\\
\lim\limits_{\stackrel{(\lambda_1,\lambda_2)
\rightarrow (0,0)}{\lambda_1\not=0,\lambda_2 \not=0}} d'd''\Big(c_1(\mathscr L_2,\|\ \|_2)\wedge G_{\lambda_1}^{s_1}\Big)= 
c_1(\mathscr L_2,\|\ \|_2)\wedge\big( -[{\rm div}(s_1)] + c_1(\mathscr{L}_1,\|\ \|_1)\big).
\end{multline}

Il r\'esulte  aussi de la proposition \ref{prop2fonctions} que dans chacun des deux ouverts $U_{s_j}$, $j=1,2$,
on a
\begin{equation}\label{eqgreen3}
\begin{split}
& \lim\limits_{\stackrel{(\lambda_1,\lambda_2)
\rightarrow (0,0)}{\lambda_1\not=0,\lambda_2 \not=0}} \big(d'd'' (G_{\lambda_1,\lambda_2}^{s_1,s_2}\big)\big) = \big(c_1(\mathscr{L}_2,\|\ \|_2)
- [{\rm div}(s_2)]\big)\wedge \big([{\rm div}(s_1)] - c_1(\mathscr{L}_1,\|\ \|_1)\big) \\
& \qquad = - [{\rm div}(s_1)] \wedge [{\rm div}(s_2)] - \big[c_1(\mathscr{L}_1,\|\ \|_1) \wedge c_1(\mathscr{L}_2,\|\ \|_2)\big] \\
& \qquad \qquad \qquad  + c_1(\mathscr{L}_1,\|\ \|_1) \wedge [{\rm div}(s_2)] + c_1(\mathscr{L}_2,\|\ \|_2) \wedge [{\rm div}(s_1)].
\end{split}
\end{equation}
Du fait de la possibilit\'e de d\'ecomposer $\langle G_{\lambda_1,\lambda_2}^{s_1,s_2},\omega \rangle$ sous 
la forme \eqref{green1aux} suivant une partition de 
l'unit\'e subordonn\'ee \`a un recouvrement de l'adh\'erence d'un voisinage ouvert de ${\rm Supp}(\omega)$ 
par des ouverts dans lesquelles une des sections $s_j$ au 
moins est r\'eguli\`ere et inversible, cette relation asymptotique entre courants est valide dans $U$ tout entier. 
En combinant \eqref{eqgreen2} et \eqref{eqgreen3} et en tenant compte de \eqref{prod2diviseurs},  
on obtient bien la relation asymptotique \eqref{green1} voulue. 
\end{proof}

\vskip 1mm
\noindent
Par r\'ecurrence sur l'entier $p=2,...,n$, nous sommes en mesure de d\'emontrer le r\'esultat suivant, pendant naturel du th\'eor\`eme 
\ref{theorempfonctions}.

\begin{theorem}\label{theorempfonctions2}
Soient $\mathscr{L}_j\rightarrow U$, $j=1,...,p$, $p\geq 2$ fibr\'es en droites au-dessus d'un ouvert $U$ d'un bon $\K$-espace analytique 
$X$ au sens de Berkovich sans bord, \'equip\'e chacun d'une m\'etrique lisse
$\|\ \|_j$. Pour chaque $j=1,...,p$, soit $s_j$ une section m\'eromorphe du fibr\'e $\mathscr{L}_j$ dans $U$. On suppose que pour tout
$1\leq j_1 < \dots < j_k \leq p$ (avec $k=1,...,p$) on a
${\rm codim}_{U} \Big(\bigcap_1^k {\rm Supp}\, \big([{\rm div}(s_{j_\ell})]\big)\Big)\geq k$ comme au th\'eor\`eme 
{\rm \ref{theorempfonctions}}.
Pour tout $(\lambda_1,...,\lambda_p) \in (\R^*)^p$, on peut d\'efinir l'action d'un courant 
$G^{s_1,...,s_p}_{\lambda_1,...,\lambda_p}$ de 
$\mathscr{D}_{n-p+1,n-p+1}(U)$ par
$$
G_{\lambda_1,...,\lambda_p}^{s_1,...,s_p}~:
\omega \in \mathscr{A}^{n-p+1,n-p+1}_c(U)
\longmapsto - \int_{U} \frac{\|s_p\|_p^{\lambda_p}}{\lambda_p}\,
\bigwedge\limits_{j=1}^{p-1} d'd'' \Big(\frac{\|s_j\|^{\lambda_j}_j}{\lambda_j}\Big)\wedge \omega
$$
apr\`es avoir d\'ecoup\'e cette int\'egrale suivant un partitionnement de l'unit\'e 
$1=\sum_\iota \varphi_\iota$ subordonn\'ee au support de $d''\omega$ afin d'en assurer la convergence. 
De plus on a, au sens de la convergence faible des courants sur $U$,
\begin{equation}\label{greentheo}
\begin{split}
& \lim\limits_{\stackrel{(\lambda_1,...,\lambda_p)
\rightarrow (0,...,0)}{\lambda_1\not=0,...,\lambda_p \not=0}}
\Big(d'd'' \Big(G^{s_1,...,s_p}_{\lambda_1,...,\lambda_p} + \sum\limits_{k=1}^{p-1}\sum\limits_{1\leq j_1<\dots < j_k\leq p}
\Big(\bigwedge_{j \not= j_1,...,j_k} c_1(\mathscr{L}_{j},\|\ \|_{j})\Big)
\wedge G^{s_{j_1},...,s_{j_k}}_{\lambda_{j_1},...,\lambda_{j_k}}\Big)\Big) \\
& \qquad + \bigwedge\limits_{j=1}^p [{\rm div}(s_j)] = \Big[\bigwedge\limits_{j=1}^p
c_1(\mathscr{L}_j,\|\ \|_j)\Big],
\end{split}
\end{equation} 
o\`u le produit de courants $[{\rm div}(s_1)] \wedge \dots \wedge  [{\rm div}(s_p)]$ est  
d\'efini localement comme l'est le courant $[{\rm div}(f_1)]\wedge \dots \wedge [{\rm div}(f_p)]$ 
dans le th\'eor\`eme \ref{theorempfonctions} \`a partir des fonctions m\'eromorphes coordonn\'ees $f_1,...,f_p$ des $s_j$ dans les rep\`eres locaux pour les fibr\'es $\mathscr {L}_j$, $j=1,...,p$.   
\end{theorem}
\begin{proof}
La preuve est calqu\'ee sur celle du th\'eor\`eme \ref{theorempfonctions}.
Le r\'esultat est acquis pour $p=2$ d'apr\`es la proposition \ref{prop2fonctions2}.
On suppose donc le r\'esultat acquis pour $p-1$ fibr\'es en droites ($p\geq 3$). On note, pour $j=1,...,p$, $Z_j$  
les sous-espaces analytiques ferm\'es (au sens de Zariski) de $U$ de codimension $1$ 
d\'efinis comme les supports des courants $[{\rm div}(s_j)]$. Pour chaque $j=1,...,p$, on note $\widehat{Z}_j$ l'intersection des 
sous-ensembles $\K$-analytiques $Z_\ell$ pour $\ell=1,...,j-1,j+1,...,p$. On note $U_{s_j}$
le plus grand ouvert de $U$ dans lequel la section $s_j$ est localement
r\'eguli\`ere et inversible. Soit $\omega \in \mathscr{A}_c^{n-p+1,n-p+1}(U)$. 
On est maintenant en mesure de donner le sens suivant \`a l'expression (en tenant compte de la d\'emarche conduisant \`a 
\eqref{limitepfonctions}) 
\begin{equation}\label{expressionpsecctions} 
- \int_{ U} d'\Big(\frac{\|s_p\|_p^{\lambda_p}}{\lambda_p}\Big) 
\wedge \Big(\bigwedge\limits_{j=1}^{p-1} d'd'' \Big(\frac{\|s_j\|_j^{\lambda_j}}{\lambda_j}\Big)\Big)\wedge d''(\varphi_\iota \omega)
\end{equation} 
suivant que le support de $\varphi_\iota$ est inclus dans l'un des $U_{s_j}$ pour $j=1,...,p-1$ ou que le support de $\varphi_\iota$ 
est inclus dans $U_{s_p}$.
\begin{itemize}
\item Si ${\rm Supp}\, \varphi_\iota \subset U_{f_{j_0}}$ pour un indice $j_0$ entre $1$ et $p-1$ et si 
${\rm codim}_{U} \widehat { Z_j} = p-1$, on peut consid\'erer
$\widehat {Z_j}$ comme un $\K$-espace analytique de dimension $n-p+1$,  on d\'efinit l'expression \eqref{expressionpsecctions}  par 
\begin{multline}\label{expressionpsectionscas1}
- \int_{U} d'\Big(\frac{\|s_p\|_p^{\lambda_p}}{\lambda_p}\Big) 
\wedge \Big(\bigwedge\limits_{j=1}^{p-1} d'd'' \Big(\frac{\|s_j\|_j^{\lambda_j}}{\lambda_j}\Big)\Big)\wedge d''(\varphi_\iota \omega)\\ 
:= 
- \Big\langle d'
\Big( \Big[\frac{\|s_p\|_p^{\lambda_p}}{\lambda_p}\Big]\, d'd''\,
G_{\lambda_1,...,\widehat{\lambda_{j_0}},...,\lambda_{p-1}}^{s_1,...,\widehat{s_{j_0}},...,s_{p-1}}\Big)\,,\, d'd''\Big(
\frac{\|s_{j_0}\|_{j_0}^{\lambda_{j_0}}}{\lambda_{j_0}}\Big) \wedge d''(\varphi_\iota \omega)\Big\rangle \\ 
= - \lambda_{j_0} 
\, \Big\langle 
d'
\Big( \Big[\frac{\|s_p\|_p^{\lambda_p}}{\lambda_p}\Big]\, d'd''\,
G_{\lambda_1,...,\widehat{\lambda_{j_0}},...,\lambda_{p-1}}^{s_1,...,\widehat{s_{j_0}},...,s_{p-1}}\Big)\,,\,  
\|s_{j_0}\|_{j_0}^{\lambda_{j_0}}\, d'(\log \|s_{j_0}\|_{j_0}) \wedge d''(\log \|s_{j_0}\|_{j_0})
\wedge d''(\varphi_\iota\omega) \Big\rangle,  
\end{multline}
et (d'apr\`es l'hypoth\`ese de r\'ecurrence) on a aussi 
$$
\big\langle G_{\lambda_1,...,\lambda_p}^{s_1,...,s_p}\,,\,
\varphi_\iota \omega \big\rangle :=
\Big\langle
G_{\lambda_1,...,\widehat{\lambda_{j_0}},...,\lambda_p}^{s_1,...,\widehat{s_{j_0}},...,s_p},
\varphi_\iota \omega \wedge d'd'' \Big(\frac{\|s_{j_0}\|_{j_0}^{\lambda_{j_0}}}{\lambda_{j_0}}\Big)
\Big\rangle.
$$
\item  Si ${\rm Supp}\, \varphi_\iota \subset U_{s_p}$, 
on d\'efinit l'expression \eqref{expressionpfonctions} par 
\begin{multline}\label{expressionpsectionscas2}
- \int_{U} d'\Big(\frac{\|s_p\|_p^{\lambda_p}}{\lambda_p}\Big) 
\wedge \Big(\bigwedge\limits_{j=1}^{p-1} d'd'' \Big(\frac{\|s_j\|_j^{\lambda_j}}{\lambda_j}\Big)\Big)\wedge d''(\varphi_\iota \omega)\\ 
:= 
- \Big\langle d'd''\,
G_{\lambda_1,...,\lambda_{p-1}}^{s_1,...,s_{p-1}}\,,\, 
d'\Big(\frac{\|s_p\|_p^{\lambda_p}}{\lambda_p}\Big)\wedge d''(\varphi_\iota\omega)\Big\rangle,
\end{multline}
et toujours suivant l'hypoth\`ese de r\'ecurrence 
$$
\big\langle G_{\lambda_1,...,\lambda_p}^{s_1,...,s_p}\,,\,
\varphi_\iota \omega \big\rangle :=
\Big\langle
d'd'' G_{\lambda_1,...,\lambda_{p-1}}^{s_1,...,s_{p-1}}\,,\,
\varphi_\iota \omega \, \frac{\|s_{p}\|_p^{\lambda_{p}}}{\lambda_{p}}
\Big\rangle.
$$
\end{itemize} 
On d\'efinit l'action du courant $G_{\lambda_1,...,\lambda_p}^{s_1,...,s_p}$
en exploitant le partitionnement de l'unit\'e (par des ouverts tous inclus dans au moins un
$U_{s_j}$) de l'adh\'erence d'un voisinage ouvert du support de $\omega$~:
$$
\big\langle G_{\lambda_1,...,\lambda_p}^{s_1,...,s_p}\,,\,
\omega \big\rangle =
\sum_\iota \big\langle G_{\lambda_1,...,\lambda_p}^{s_1,...,s_p}\,,\,
\varphi_\iota\, \omega \big\rangle.
$$
Il r\'esulte du th\'eor\`eme \ref{theorempfonctions} et  des \'egalit\'es \eqref{eqgreen3}, \eqref{expressionpsectionscas1} et 
\eqref{expressionpsectionscas2} que dans chaque ouvert
$U_{s_j}$ ($j=1,...,p$), on a, pour la convergence au sens de la limite faible des courants dans $U$,
\begin{equation}\label{green2}
\begin{split}
& \lim\limits_{\stackrel{(\lambda_1,...,\lambda_p)
\rightarrow (0,...,0)}{\lambda_1\not=0,...,\lambda_p \not=0}}
\Big(d'd'' (G^{s_1,...,s_p}_{\lambda_1,...,\lambda_p})\Big)
= -\bigwedge\limits_{j=1}^p \Big( [{\rm div}(s_j)] - c_1(\mathscr{L}_j,\|\ \|_j)\Big) \\
& = - [{\rm div}(s_1)] \wedge \dots \wedge [{\rm div}(s_p)] + \\
& + \sum\limits_{k=1}^{p-1}
\sum\limits_{1\leq j_1 < \dots < j_k\leq p}
(-1)^{p-1-j} \Big(\bigwedge\limits_{\ell =1}^k
[{\rm div}(s_{j_\ell})]\Big)
\wedge \Big(\bigwedge\limits_{j\not = j_1,...,j_k}
c_1(\mathscr{L}_j,\|\ \|_j)\Big),
\end{split}
\end{equation}
avec
\begin{multline*}
-\bigwedge\limits_{j=1}^p \Big( [{\rm div}(s_j)] - c_1(\mathscr{L}_j,\|\ \|_j)\Big)~:=\\
\begin{cases}
 \Big(c_1(\mathscr L_p,\|\ \|_p)-[{\rm div}(s_p)]\Big)\bigwedge\limits_{j=1}^{p-1} \Big( [{\rm div}(s_j)]
 - c_1(\mathscr{L}_j,\|\ \|_j)\Big) \mbox{ dans } U_{s_p}\\
 \Big(c_1(\mathscr L_{j_0},\|\ \|_{j_0})-[{\rm div}(s_{j_0})]\Big)\bigwedge\limits_{\stackrel{ j=1}{j\neq j_0}}^p \Big( [{\rm div}(s_j)]
 - c_1(\mathscr{L}_j,\|\ \|_j)\Big) \mbox{ dans } U_{s_{j_0}}.
\end{cases}
\end{multline*}
La formule asymptotique \eqref{green2} est donc valide au sens des courants dans $U$ puisque l'on peut utiliser un partitionnement de 
l'unit\'e subordonn\'e au 
recouvrement d'une forme test ${\rm Supp}\, \omega$ par les $U_{s_j}$. Pour chaque valeur de $k$ entre $1$ et $p-1$, pour 
chaque suite de $k$ indices distincts
$1\leq j_1< \dots < j_k \leq p$, on substitue au second membre de la relation \eqref{green2} les relations asymptotiques
\begin{equation*}
\begin{split}
& \bigwedge\limits_{\ell = 1}^k [{\rm div}(s_k)] = \Big[\bigwedge_{\ell =1}^k c_1(\mathscr{L}_{j_\ell},\|\ \|_{j_\ell})\Big] \\
& -
\lim\limits_{\stackrel{(\lambda_{j_1},...,\lambda_{j_k})
\rightarrow (0,...,0)}{\lambda_{j_1}\not=0,...,\lambda_{j_k} \not=0}}
\Big(d'd'' \Big(G^{s_{j_1},...,s_{j_k}}_{\lambda_{j_1},...,\lambda_{j_k}} \\
& \qquad \qquad \qquad \qquad \qquad \quad
 + \sum\limits_{\kappa=1}^{k-1}\sum\limits_{1\leq \iota_1<\dots < \iota_\kappa\leq k}
\Big(\bigwedge_{\iota \not= \iota_1,...,\iota_{\kappa}} c_1(\mathscr{L}_{j_\iota},\|\ \|_{j_\iota})\Big)
\wedge G^{s_{j_{\iota_1}},...,s_{j_{\iota_k}}}_{\lambda_{j_{\iota_1}},...,\lambda_{j_{\iota_{\kappa}}}}\Big)\Big)
\end{split}
\end{equation*}
avant de regrouper dans le membre de gauche de \eqref{green2} ainsi transform\'e tous les termes s'exprimant comme des limites 
(et devant lesquels figure l'action de 
l'op\'erateur de Green $d'd''$).
\end{proof}

\noindent
Le th\'eor\`eme \ref{theorempfonctions2} est \`a rapprocher de la construction de courants de Green normalis\'es inspir\'ee par 
la m\'ethode de prolongement analytique, 
telle qu'elle est par exemple d\'ecrite dans \cite[section 3]{BY98}. On note que dans ce nouveau cadre on dispose de $p$ param\`etres $\lambda_1,...,\lambda_p$ (au lieu 
d'un seul, comme dans \cite[proposition 4]{BY98}) pour construire une solution $G^{s_1,...,s_p}_{\lambda_1,...,\lambda_p}$ \`a une approximation de l'\'equation de Green 
normalis\'ee
\begin{equation}\label{greennorm}
d'd'' G + \bigwedge\limits_{j=1}^p [{\rm div}(s_j)] = \Big[\bigwedge\limits_{j=1}^p
c_1(\mathscr{L}_j,\|\ \|_j)\Big].
\end{equation}
Mais il est par contre possible (dans ce cadre des espaces $\K$-analytiques au sens de Berkovich) de supposer les sections $s_j$ m\'eromorphes et non seulement holomorphes 
comme c'\'etait le cas dans le cadre analytique complexe~; le fait que toute $(\ell,k)$-forme lisse \`a support compact sur un bon espace analytique $Y$ 
de dimension $k$ soit telle que son support \'evite tout ferm\'e de Zariski d'int\'erieur non vide de $Y$
(voir \cite[lemme 3.2.5]{ChLD}) joue dans ce cadre non archim\'edien un r\^ole majeur.
Par contre, il convient de faire, lorsque l'on travaille dans un tel cadre, une hypoth\`ese plus forte concernant les supports des diviseurs que celle consistant \`a juste 
supposer que ces supports s'intersectent proprement~; il est n\'ecessaire en effet de supposer que c'est aussi le cas pour toute sous-famille extraite de la famille des 
supports des $s_j$, $j=1,...,p$.
\vskip 1mm
\noindent
Pour construire une solution $G$ \`a l'\'equation de Green normalis\'ee
\eqref{greennorm} (et non seulement une solution \`a une approximation de cette \'equation
suivant \eqref{greentheo}), il convient par exemple de complexifier le $\R$-espace vectoriel
$\mathscr{D}_{n-p+1,n-p+1}(U)$ et de former, dans ce complexifi\'e $\mathscr{D}_{n-p+1,n-p+1}(U)\otimes_\R \C$, le courant
\begin{equation}\label{greenexplicit}
\begin{split}
& G^{s_1,...,s_p} := \frac{1}{(2i\pi)^p} \times \\
& \int_{\Gamma_{r_1,...,r_p}}
\Big(G^{s_1,...,s_p}_{\lambda_1,...,\lambda_p} +
\sum\limits_{k=1}^{p-1}\sum\limits_{1\leq j_1<\dots < j_k\leq p}
\Big(\bigwedge_{j \not= j_1,...,j_k} c_1(\mathscr{L}_{j},\|\ \|_{j})\Big)
\wedge G^{s_{j_1},...,s_{j_k}}_{\lambda_{j_1},...,\lambda_{j_k}}\Big)\, \bigwedge_1^p\frac{d\lambda_j}{\lambda_j}
\end{split}
\end{equation}
o\`u $r_1,...,r_p>0$,
$$
\Gamma_{r_1,...,r_p}~:
(t_1,...,t_p)\in [0,1]^p
\mapsto(r_1 e^{2i\pi t_1},...,r_p e^{2i\pi t_p}) = (\lambda_1,...,\lambda_p).  
$$
Il est en effet possible
de supposer dans les th\'eor\`emes \ref{theorempfonctions} et \ref{theorempfonctions2} que les param\`etres $\lambda_1,...,\lambda_p$ sont dans $\C^*$ et non plus dans $\R^*$.
Le courant \og moyen\fg\ $G^{s_1,...,s_p}$ ainsi construit est un courant r\'eel car $\overline{G^{s}_\lambda + \cdots}
= G^s_{\bar \lambda} + \cdots$ et que la forme 
$\Gamma_{r_1,...,r_p}^* \big(\bigwedge d\lambda_j/(2i\pi \lambda_j)\big)$ est la forme  r\'eelle $\bigwedge_j \big(d\theta_j/(2\pi)\big)$. 
Ce courant d\'epend naturellement de l'ordre dans lequel sont consid\'er\'es les fibr\'es $\mathscr{L}_1,...,\mathscr{L}_p$ et les sections m\'eromorphes qui y sont 
attach\'ees. Il r\'esulte des th\'eor\`emes \ref{theorempfonctions} et
\ref{theorempfonctions2} (repris en supposant cette fois les $\lambda_j$ dans $\C^*$) que le courant $G^{s_1,...,s_p}$ est solution de l'\'equation de Green normalis\'ee
\eqref{greennorm}.
\section{Approche du type Mellin aux courants de Vogel dans le cadre alg\'ebrique}\label{sectionVogel}
Dans cette section, nous nous pla\c cons dans le cadre alg\'ebrique et consid\'erons une vari\'et\'e alg\'ebrique projective $X$ de dimension $n$ d\'efinie au-dessus du corps 
valu\'e $\K$, un entier $m\in \N^*$, et la vari\'et\'e alg\'ebrique projective produit $\P^m_\K \times X$ de dimension $n+m$. On se donne un fibr\'e en droites 
$L_X \rightarrow X$ au-dessus de $X$ et des sections globales $s_0,...,s_m$ du fibr\'e $L_X$ au-dessus de $X$. Comme  le foncteur d'analytification est compatible 
avec le produit fibr\'e, on a $\big(\P^m_\K\times X\big)^{\rm an} =
(\P^m_\K)^{\rm an} \times X^{\rm an}$.
\vskip 1mm
\noindent
Soit $\|\ \|$ une m\'etrique semi-positive sur le fibr\'e en droites $\mathcal O_{\P^m_\K} (1) \rightarrow \P^m_\K$. 
On sait (voir \cite{ChL06, ChL11, Gub08, BFJ},  \cite[section 6.9]{ChLD} ou aussi le survey
\cite[section 3.3]{Yuan}) lui associer une mesure
de Monge-Amp\`ere que l'on note en effet $\big(c_1(\mathcal O_{\P^m(\K)}(1),\|\ \|)\big)^{\wedge^m}$
sur l'analytification $(\P^m_\K)^{\rm an}$ telle que
$$
\int_{(\P^m_\K)^{\rm an}} \big(c_1(\mathcal O_{\P^m_\K}(1),\|\ \|)\big)^{\wedge^m}(\kappa) =
\deg_{\mathcal O_{\P^m_\K}(1)} (\P^m_\K) = 1.
$$
\vskip 1mm
\noindent
Lorsque le fibr\'e ainsi m\'etris\'e $\overline{\big(\mathscr O_{\P^m_\K}(1)\big)^{\rm an}}$ est un fibr\'e vectoriel PL
(voir \cite[d\'efinition 6.2.9]{ChLD}, ceci signifiant essentiellement que l'on puisse disposer localement de rep\`eres orthonorm\'es), la mesure de Monge-Amp\`ere 
$\big(c_1(\mathcal O_{\P^m_\K}(1),\|\ \|)\big)^{\wedge^m}$ est une mesure atomique support\'ee par un sous-ensemble discret $S_{\|\ \|}$, i.e il 
existe des r\'eels positifs $\gamma_\eta$ tels que pour toute fonction $\varphi$ continue de $(\P^m_\K)^{\rm an}$ dans $\R$
(\cite{ChLD}, proposition 6.9.2 et
d\'efinition 6.7.2 pour la d\'efinition de $S_{\|\ \|}$)

\begin{equation}\label{dirac}
\int_{(\P^m_\K)^{\rm an}} \varphi(\kappa) \,
\big(c_1(\mathcal O_{\P^m_\K}(1),\|\ \|_{\rm moy})\big)^{\wedge^m}(\kappa) =
\sum\limits_{\eta\in S_{\|\ \|}} \gamma_\eta\,  \varphi(\eta).
\end{equation}
On supposera par la suite que l'on est toujours dans cette situation (m\'etrique $\|\ \|$ semi-positive et fibr\'e m\'etris\'e 
$\overline{\big(\mathscr O_{\P^m_\K}(1)\big)^{\rm an}}$ PL)~;
si la m\'etrique n'est plus semi-positive mais que le fibr\'e m\'etris\'e $\overline{\big(\mathscr O_{\P^m_\K}(1)\big)^{\rm an}}$ est toujours PL, 
les masses $\gamma_\eta$ dans \eqref{dirac} sont
des nombres r\'eels non n\'ecessairement positifs ou nuls.

\begin{example} {\rm Dans le cas particulier o\`u $\|\ \|$ d\'esigne la m\'etrique standard
$$
\|\langle \kappa,z\rangle\|_{\rm std} = \frac{|\langle \kappa,z\rangle|}{\max (|z_0|,...,|z_m|)},\quad \kappa \in \K^{m+1}\setminus \{(0,...,0)\},\quad
z =[z_0:\dots : z_m]
$$
(qui est bien semi-positive, se r\'ef\'erer par exemple \`a la section 1.3 de \cite{ChL11}),
la mesure de Monge-Amp\`ere $\big(c_1(\mathscr O_{\P^m_\K}(1),\|\ \|)\big)^{\wedge^m}$
sur $(\P^m_\K)^{\rm an}$ qui lui est attach\'ee est la mesure de Dirac $\delta_\xi$ au point de Gau\ss.
}
\end{example}

\begin{remark}{\rm Dans le cadre archim\'edien ($\K=\C$), la m\'etrique sur $\P^m_\C$ construite
sur le m\^eme principe que celui sur lequel est construite $\big(c_1(\mathcal O_{\P^m_\K}(1),\|\ \|)\big)^{\wedge^m}$
s'obtient comme image directe de la mesure de Haar normalis\'ee sur le tore
$$
\{[z_0:\dots : z_m] \in \P^m_\C\,;\, |z_0| = \dots = |z_m|\}.
$$
Notons que la m\'etrique $\|\ \|_{\rm std}$ est continue mais non lisse.
Toujours dans ce cadre archim\'edien, mais lorsque la m\'etrique $\|\ \|$ est la m\'etrique de Fubini-Study (qui, elle, est lisse)
$$
\|\langle \kappa,z\rangle\|_{\rm fs} = \frac{|\langle \kappa,z\rangle|}
{\sqrt{|z_0|^2+\dots + |z_m|^2}},\quad \kappa \in \C^{m+1}\setminus \{(0,...,0)\},\quad
z =[z_0:\dots : z_m],
$$
on obtient naturellement $\big(c_1(\mathscr O_{\P^m_\C}(1),\|\ \|_{\rm fs})\big)^{\wedge^m} =
(dd^c \log \|z\|^2)^{\wedge^m}$, m\'etrique pour laquelle on rappelle que l'on dispose de la formule de Crofton~: si $f_0,...,f_m$ sont $m+1$ \'el\'ements de 
$\mathcal O_{X}(U)$ (o\`u $U$ d\'esigne un ouvert d'un espace analytique complexe
$X$), on a, lorsque les $f_j$ n'ont aucun z\'ero commun dans $U$~:
\begin{equation}\label{crofton}
dd^c (\log \|f(x)\|^2_{\rm eucl}) = \int_{[\kappa_0:\cdots : \kappa_m] \in \P^m_\C}
\big[{\rm div}(\langle \kappa,f(x)\rangle\big] \wedge
\big(c_1(\mathcal O_{\P^m_\C},\|\ \|_{\rm fs})\big)^{\wedge^m}(\kappa),
\end{equation}
$\|\ \|_{\rm eucl}$ d\'esignant la norme euclidienne sur $\C^{n+1}$
(voir par exemple \cite{ASWY14}, lemme 6.3) et $f=(f_0,...,f_m)$.}
\end{remark}
\vskip 2mm
\noindent
Dans le cadre non archim\'edien (alg\'ebrique), nous pouvons \'enoncer ce qui peut \^etre consid\'er\'e comme le pendant
de la formule de Crofton. On consid\`ere les analytifications $\big(\mathcal O_{\P^m_\K}(1)\big)^{\rm an}$ et $\mathscr{L}_X^{\rm an}$ respectivement des fibr\'es en 
droites $\mathcal O_{\P^m_\K}(1)$
et $L_X$ (consid\'er\'es tous deux comme des fibr\'es en droites au-dessus de la vari\'et\'e
alg\'ebrique projective produit $\P^m_\K\times X$) et la section du fibr\'e produit
$\big(\mathcal O_{\P^m_\K}(1)\big)^{\rm an}\otimes \mathscr{L}_X^{\rm an}$ obtenue en analytifiant la section $(\kappa,z) \mapsto \langle \kappa,f(z)\rangle$ du fibr\'e
en droites produit $\mathcal O_{\P^m_\K}(1) \otimes L_X$. On notera $[{\rm div}(\langle \kappa,f\rangle)]$ le courant d'int\'egration correspondant \`a ce diviseur effectif 
sur $\big(\P^{m}_\K \times X\big)^{\rm an} = (\P^m_\K)^{\rm an}
\times X^{\rm an}$. On suppose ici que le fibr\'e m\'etris\'e
$\overline{\big(\mathcal O_{\P^m_\K}(1)\big)^{\rm an}}$ est PL. La formule de Crofton
\eqref{crofton} dans ce cadre non archim\'edien s'\'enonce alors ainsi : \'etant donn\'ees des sections holomorphes $s_0,...,s_m$ de $L_X$ telles que 
$\bigcap_1^m {\rm Supp}([{\rm div}(s_j)]) = \emptyset$ et $s^{\rm an}_j$ ($j=0,...,m$) leurs analytifications,
on a\footnote{Il faut comprendre ici $\langle \kappa,s^{\rm an}\rangle$ comme l'analytification de la section $\langle \kappa,s\rangle$ du fibr\'e
$\mathcal O_{\P^m_\K} (1) \otimes L_X\rightarrow \P^m_\K\times X$ en une section du fibr\'e
$\big(\mathcal O_{\P^m_\K}(1)\big)^{\rm an}\otimes \mathscr{L}_X^{\rm an} \rightarrow (\P^m_\K)^{\rm an} \times X^{\rm an}$.}~:
\begin{equation}\label{croftonarch}
\begin{split}
& d'd''\Big(\sum\limits_{\eta \in S_{\|\ \|}} \lambda_\eta\, \big[\log \|\langle \eta,s^{\rm an}(x)\rangle\|\big]\Big) \\
& = \int_{\kappa \in (\P^m_\K)^{\rm an}} \big[{\rm div} (\langle \kappa,s^{\rm an}(x)\rangle)\big]\wedge
\big(c_1(\mathcal O_{\P^m_\K}(1),\|\ \|)\big)^{\wedge^m}(\kappa).
\end{split}
\end{equation}
Elle se r\'eduit pour la m\'etrique standard \`a l'\'equation de Lelong-Poincar\'e
$$
d'd'' \big(\big[\log \|\langle \xi,s^{\rm an}(x)\rangle\|_{\rm sdt}\big]\big) =
\big[{\rm div}\langle \xi,s^{\rm an}(x)\rangle\big],
$$
o\`u $\xi\in (\P^m_\K)^{\rm an}$ d\'esigne le point de Gau\ss.
\vskip 2mm
\noindent
Soit $\pi~:\widehat{X}^{\rm an}\to X^{\rm an}$ un \'eclatement normalis\'e de $X^{\rm an}$ (\cite{Con99}, commentaire avant le lemme 2.2.1). 
Les composantes irr\'eductibles de $X^{\rm an}$ sont des sous-ensembles analytiques de la forme $U_i=\pi(\widehat{U_i})$, avec $\widehat{U_i}$ les composantes connexes de l'\'eclatement normalis\'e $\pi~:\widehat{X}^{\rm an}\to X^{\rm an}$. On dit que $X^{\rm an}$ est irr\'eductible si il est non vide et admet une unique composante irr\'eductible (\cite[lemme 2.2.1 et d\'efinition 2.2.2]{Con99}).
\vskip 1mm
\noindent
Pour d\'efinir une approche de type Mellin au cycle de Vogel attach\'e \`a une famille
$(s_0,...,s_m)$ de sections d'un fibr\'e en droites $L_X \rightarrow X$, une fois choisie
une m\'etrique lisse $\|\ \|_{L_X}$ sur le fibr\'e $\mathscr{L}_X^{\rm an} \rightarrow X^{\rm an}$, il suffit d'exploiter de mani\`ere it\'erative le lemme suivant, 
directement inspir\'e de \cite[lemme 3.1]{ASWY14}.

\begin{lemma}\label{lemmeintersection} Soit $U$ un ouvert   d'un bon $\K$-espace analytique au sens de Berkovich
$X^{\rm an}$ de dimension $n$,
$$
Z = \sum\limits_{\iota} \mu_\iota\, Z_\iota
$$
une combinaison formelle localement finie de sous-ensembles analytiques de $U$ de dimension pure $n-p$ ($1\leq p \leq n-1$) et $s$ une section holomorphe
d'un fibr\'e m\'etris\'e $\mathscr{L}^{\rm an}\rightarrow U$ \'equip\'e d'une m\'etrique lisse $\|\ \|$ de premi\`ere forme de Chern 
$c_1(\mathscr{L}^{\rm an},\|\ \|)$.
On note
\begin{equation*}
\begin{split}
& Z^{{\rm div}(s)} := \sum\limits_{\big\{\iota\,;\,
{\rm Supp} (Z_\iota)\, \subset\, {\rm Supp} ({\rm div}(s))\big\}}
\mu_\iota  Z_\iota \\
& Z^{U \setminus {\rm div}(s)} := \sum\limits_{\big\{\iota\,;\,
{\rm Supp} ( Z_\iota)\, \not\subset\, {\rm Supp} ({\rm div}(s))\big\}}
\mu_\iota Z_\iota.
\end{split}
\end{equation*}
Soit $\lambda \in \{\lambda \in \C\,;\, {\rm Re}\lambda >0\}$.
On d\'efinit $\tilde T^s_\lambda \in \big(\mathscr{D}_{n-p,n-p}(U) \oplus  \mathscr{D}_{n-p-1,n-p-1}(U)\big)\otimes_\R \C$ comme
\begin{equation}
\tilde T^s_\lambda := \sum\limits_{\iota} 
\Big([1 - \|s\|^{\lambda}] + \big[\|s\|^\lambda \ c_1(\mathscr{L}^{\rm an},\|\ \|)\big] + d'd'' \Big[\frac{\|s\|^\lambda}{\lambda}\Big] \Big) \wedge [Z_\iota], 
\end{equation}
o\`u le courant $[\|s\|^\lambda]\, [Z_\iota]$ est d\'efini \`a partir du lemme 
4.6.1 de \cite{ChLD} comme l'image directe par $i_{Z_\iota}~: Z_\iota \rightarrow U$
du courant $[\|s\circ i_{Z_\iota}\|^\lambda]$. 
On a
$$
\tilde T^s_\lambda =  \big[Z^{{\rm div}(s)}\big] +\big[\|s\|^\lambda\,
c_1(\mathscr{L}^{\rm an},\|\ \|)\big] \wedge \big[Z^{U \setminus {\rm div}(s)}\big]  +
d'd'' \big(\Big[\frac{\|s\|^\lambda}{\lambda}\Big]\big)\wedge [Z^{U\setminus {\rm div}(s)}]
$$
et, par cons\'equent :
\begin{equation}\label{vogel1}
\lim\limits_{\stackrel{\lambda \rightarrow 0}{\rm Re\, \lambda >0}}
\tilde T^s_\lambda = \big[Z^{{\rm div}(s)}\big] +
[{\rm div}(s)] \wedge \big[Z^{U \setminus {\rm div}(s)}\big].
\end{equation}
\end{lemma}

\begin{proof} Ce lemme r\'esulte imm\'ediatement de l'\'equation de Lelong Poincar\'e.
\end{proof}

\noindent
Suivant l'approche propos\'ee dans \cite{ASWY14} (voir en particulier le th\'eor\`eme 6.2 dans cette r\'ef\'erence) et la transcription \eqref{croftonarch} 
que nous avons propos\'e pour la formule de Crofton dans le cadre non archim\'edien, il est naturel de d\'efinir ainsi le courant de Vogel (et son approche 
du type Mellin) attach\'e \`a $m+1$ sections globales $s_0,...,s_m$ d'un fibr\'e $L_X \rightarrow X$ au-dessus d'une vari\'et\'e projective $X$ d\'efinie 
au-dessus d'un corps valu\'e $\K$, une fois choisie une m\'etrique lisse
$\|\ \|_{L_X}$ sur le fibr\'e $\mathscr{L}^{\rm an}$. On note
$\|\ \|_{L_X,{\rm moy}}$ la m\'etrique induite sur le fibr\'e
$\big(\mathcal O_{\P^m_\K}(1)\big)^{\rm an}\otimes \mathscr{L}_X^{\rm an}$ par le choix des
m\'etriques $\|\ \|$ sur $\big(\mathcal O_{\P^m_\K}(1)\big)^{\rm an}$ et
$\|\ \|_{L_X}$ sur $\mathscr{L}^{\rm an}$.

\begin{definition} {\rm Le courant de Vogel attach\'e \`a $s_0,...,s_m$ est d\'efini comme
la limite suivante au sens (faible) des courants sur $X^{\rm an}$~:
\begin{equation}
\begin{split}
& \lim\limits_{\lambda_\nu \rightarrow 0}
\Big( \lim\limits_{\lambda_{\nu-1} \rightarrow 0} \Big(
\cdots \Big(\lim\limits_{\lambda_1 \rightarrow 0} \, \int\limits_{\big((\P^m_\K)^{\rm an}\big)^\nu}  
\Big(\bigwedge\limits_{j=1}^\nu  \big(c_1(\mathcal O_{\P^m_\K}(1),\|\ \|)\big)^{\wedge^m}(\kappa_j)\Big) \wedge \\
& \bigwedge\limits_{j=1}^\nu
\Big([1 - \|\langle \kappa_j,s^{\rm an}\rangle\|_{L_X,{\rm moy}}^{\lambda_j}]
+ \big[\|\langle \kappa_j,s^{\rm an}\rangle\|_{L_X,{\rm moy}}^{\lambda_j}\,
(c_1(L_X,\|\ \|)\big] \\
& \qquad \qquad  \qquad \qquad  \qquad \quad
+ d'd''\Big( \Big[\frac{\|\langle \kappa_j,s^{\rm an}\rangle\|_{L_X,{\rm moy}}^{\lambda_j}}{\lambda_j}\Big]\Big)\Big)\Big)
\cdots \Big)\Big)
\end{split}
\end{equation}
lorsque $\nu := \min(m+1,n+1)$, o\`u le produit des courants 
se trouve justifi\'e par le lemme 4.6.1 de \cite{ChLD} si l'on tient compte de 
\eqref{vogel1} et du fait que les limites suivant $\lambda_1,...,\lambda_\nu$ sont prises les unes apr\`es les autres. 
}
\end{definition}

\begin{remark} {\rm Du fait que la mesure correspondant aux courant $[c_1(\mathcal O_{\P^m_\K}(1),\|\ \|_{\rm moy})]^{\wedge^m}$ est une mesure atomique 
(combinaison lin\'eaire de masses de Dirac), il r\'esulte du lemme \ref{lemmeintersection} que le courant de Vogel est un courant d'int\'egration sur un 
cycle analytique (non de dimension pure) de $X^{\rm an}$ que l'on convient de d\'efinir comme le cycle (moyen) de Vogel.}
\end{remark}

\section{Approche de type Mellin aux courants de Segre dans le cadre alg\'ebrique}\label{sectionsegre}

Soit $X$ une vari\'et\'e alg\'ebrique projective de dimension $n$ d\'efinie sur $\K$ et
$X^{\rm an}$ son analytification au sens de Berkovich. On consid\`ere un fibr\'e
alg\'ebrique $E_X \rightarrow X$ de rang $m+1$ au-dessus de $X$ et on \'equipe son analytifi\'e $E_X^{\rm an} \rightarrow X^{\rm an}$  
d'une m\'etrique formelle PL (voir \cite[d\'efinition 6.2.9]{ChLD})
not\'ee $\|\ \|_{E_X^{\rm an}}$
au-dessus de l'analytification $X^{\rm an}$.

\begin{example}{\rm Si $X^{\rm an}=(\P^n_\K)^{\rm an}$ et si
$E_X^{\rm an} = \big(\mathscr O_X(d_0)\big)^{\rm an} \oplus \dots \oplus \big(\mathscr O_X(d_m)\big)^{\rm an}$, on peut
\'equiper chaque $\big(\mathscr O_X (d_j)\big)^{\rm an}$ de la m\'etrique standard
$$
\|s_j([z_0:\dots :z_n])\|_{\rm std} = \frac{|s_j(z_0,...,z_n)|}{\max_\ell |z_\ell|^{d_j}}
$$
qui est une m\'etrique globalement psh-approchable
\cite[proposition 6.3.2]{ChLD}
et le fibr\'e $E_X^{\rm an}$ de la m\'etrique
$$
\|s\|_{E_X^{\rm an}} := \max_{j} \|s_j\|_{\rm std}.
$$
}
\end{example}
\vskip 2mm
\noindent
Soit $s \in \mathscr O_X(E_X)$ une section globale de $E_X$
dont on notera $s^{\rm an}$ l'analytification. Soit $\pi~: \widehat X \rightarrow X$ l'\'eclatement normalis\'e de $X$ suivant le faisceau d'id\'eaux de $\mathcal O_X$ 
induit par
$s$ et $\pi^{\rm an}~: \widehat{X}^{\rm an} \rightarrow X^{\rm an}$ son analytification. On note $L_{\widehat X}$ le fibr\'e en droites correspondant 
au diviseur exceptionnel $D_s$ de $\pi~: \widehat X \rightarrow X$ et $\widehat{\mathscr{L}}^{\rm an}$ le fibr\'e que $L_{\widehat X}$ induit au-dessus de l'analytification 
$\widehat{X}^{\rm an}$. On a (du fait de la  d\'efinition de l'\'eclatement normalis\'e $\pi$)
$\pi^* (s) = \sigma\otimes \tau $, o\`u $\sigma$ est une section globale du fibr\'e en droites $L_{\widehat X}$ et $\tau$ une section ne s'annulant pas du fibr\'e 
$F_{\widehat X} := L_{\widehat X}^{-1} \otimes \pi^* (E_X)$ (de rang $m+1$ comme $E_X$, et dont on notera 
$F_{\widehat X}^{\rm an}\rightarrow \widehat{X}^{\rm an}$ l'analytification).
\vskip 1mm
\noindent
Comme dans la section 4 de \cite{ASWY14}, on \'equipe le fibr\'e $L_{\widehat X}$ de la m\'etrique $\|\ \|_\tau$
telle que $\|\sigma\|_\tau = \|\pi^*(s)\|_{\pi^*(E_X)}$.
On note $\sigma^{\rm an}$ et $\tau^{\rm an}$ les sections holomorphes
respectivement des fibr\'es $\widehat {\mathscr L}^{\rm an}$ et $F_{\widehat X}^{\rm an}$
au-dessus de $\widehat {X}^{\rm an}$ d\'eduites de $\sigma$ et $\tau$ par analytification. On note $\|\ \|_{\tau^{\rm an}}$ la m\'etrique formelle d\'efinie 
sur le fibr\'e $\widehat{\mathscr{L}}^{\rm an}$.
Le courant $-d'd'' \log \|\tau^{\rm an}\|_{\pi^*(E_X)}$
(calcul\'e ici localement en choisissant arbitrairement une trivialisation locale de
$(\widehat{\mathscr{L}}^{\rm an})^{-1}$) est le courant de Chern $c_1(\widehat{\mathscr{L}}^{\rm an},\|\ \|_{\tau^{\rm an}})$.
\\
Si $a_0,...,a_\mu$ sont des fonctions
r\'eguli\`eres globalement inversibles dans un ouvert $U$ de
$X^{\rm an}$, la fonction $\log \max(|a_0|,...,|a_\mu|)$ est globalement
psh approchable (voir \cite[Proposition 6.8.3]{ChLD}) dans $U$.
Comme la m\'etrique $\|\ \|_{E^{\rm an}_X}$ est suppos\'ee PL, il en est de m\^eme
pour la m\'etrique $\|\ \|_{\pi^*(E^{\rm an}_X)}$ sur
$\widehat {X}^{\rm an}$ \cite[6.2.15]{ChLD}. Par cons\'equent, la fonction
$-d'd'' \log \|\tau^{\rm an}\|_{\pi^*( E^{\rm an}_X)}$ est une fonction globalement psh-approchable au voisinage de tout point $\hat x$ o\`u $\sigma$ n'est pas 
inversible (il suffit pour cela de travailler dans un ouvert de carte $U_{\pi(\hat x)}$ au-dessus duquel on dispose d'un rep\`ere orthonorm\'e pour le fibr\'e 
$E^{\rm an}_X$ et de consid\'erer le voisinage
$\pi^{-1} (U_{\pi(x)})$ de $\hat x$)
et l'on sait donc donner un sens (en approchant cette fonction par des fonctions psh lisses) aux puissances ext\'erieures 
$\big(-c_1(\widehat{\mathscr{L}}^{\rm an},\|\ \|_{\tau^{\rm an}})\big)^{\wedge^{k-1}}$, $k=1,...,n$.
Pour $1\leq k \leq n$, on peut donc d\'efinir sur $\widehat {X}^{\rm an}$ le courant $[{\rm div}(\sigma^{\rm an})] 
\wedge \big(-c_1(\widehat{\mathscr{L}}^{\rm an},\|\ \|_{\tau^{\rm an}})\big)^{\wedge^{k-1}}$.

\noindent
En transposant la notion de courant de Segre $M^s$ introduite dans \cite[section 4]{ASWY14}, on
aboutit \`a la d\'efinition suivante~:

\begin{definition} Le courant de Segre attach\'e \`a la section $s$ est le courant
$$
M^s := [1 - \|s^{\rm an}\|^{\lambda}_{ E_X^{\rm an}}]_{\lambda = 0}
+ \pi_* \Big( \sum\limits_{k=1}^{n} [{\rm div}(\sigma^{\rm an})] \wedge
\big(-c_1(\widehat{\mathscr{L}}^{\rm an},\|\ \|_{\tau^{\rm an}})\big)^{\wedge^{k-1}}\Big).
$$
\end{definition}

\noindent
Nous avons la proposition suivante~:

\begin{prop}\label{propsegre} Le courant de Segre $M^s$ s'exprime aussi comme
$M^s = \sum\limits_{k=0}^n M^s_k$, o\`u
\begin{equation}\label{propsegre1}
\begin{split}
& M^s_0 = \lim\limits_{\lambda_0 \rightarrow 0} \big[1 - \|s^{\rm an}\|^{\lambda_0}_{ E_X^{\rm an}}\big]~; \\
& M^s_k = \lim\limits_{\lambda_k \rightarrow 0}
\Big(\lim\limits_{\lambda_{k-1} \rightarrow 0} \Big(
\cdots \Big(\lim\limits_{\lambda_1 \rightarrow 0} \\
& \Big(d''[\|s^{\rm an}\|^{\lambda_k}_{ E_X^{\rm an}}] \wedge d'[\log \|s^{\rm an}\|_{E_X^{\rm an}}] 
\wedge \bigwedge\limits_{\ell =1}^{k-1}
d'd'' \Big(\Big[\frac{\|s^{\rm an}\|_{E_X^{\rm an}}^{\lambda_\ell}}{\lambda_\ell}\Big]\Big)\Big)\Big) \cdots \Big) \Big).
\end{split}
\end{equation}
\end{prop}

\begin{proof}
La preuve est directement inspir\'ee de celle qui est conduite dans le cadre complexe dans
\cite[section 4]{ASWY14}. Puisqu'on a localement l'\'egalit\'e (au sens des courants) 
$$
[\log \|\pi^*[s^{\rm an}]\|_{\pi^*(E^{\rm n}_X)}]
= [\log |\sigma^{\{\rm an\}}|] + [\log \|\tau^{\rm an}\|] =
[\log \|\sigma^{\rm an}\|_{\|\tau\|^{\rm an}}],
$$
o\`u nous avons not\'e $\sigma^{\{\rm an\}}$ la fonction coordonn\'ee de 
$\sigma^{\rm an}$ dans un rep\`ere local, 
il d\'ecoule de la formule de Lelong-Poincar\'e que
$$
d'd'' \big[\log \|\pi^*[s^{\rm an}]\|_{\pi^*(E^{\rm an}_X)}\big] =
[{\rm div}(\sigma^{\rm an})] - c_1(\widehat {\mathscr{L}}^{\rm an},\|\ \|_{\tau^{\rm an}}).
$$
au sens des courants. Notons $M^{s,\lambda}_k$ ($k=0,...,n$) la composante de bidegr\'e
$(0,k)$ dans le courant dont on prend la limite lorsque les $\lambda_j$ tendent (les uns apr\`es les autres) vers $0$ au second membre de \eqref{propsegre1}.
On a pour $\lambda_0 >0$
$$
\pi^* (M_0^{s,\lambda})= 1 - [\|\pi^*[s^{\rm an}]\|^{\lambda_0}_{E_{X^{\rm an}}^{\rm an}}]
$$
et, pour $\lambda_k >0$ ($k=1,...,n$)~:
$$
\pi^* (M_k^{s,\lambda})
=[{\rm div}(\sigma^{\rm an})] \wedge
\big(-c_1(\widehat{\mathscr{L}}^{\rm an},\|\ \|_{\tau^{\rm an}})\big)^{\wedge^{k-1}}\,.
$$
Si l'on remplace le $(1,1)$-courant $-c_1(\widehat{\mathscr{L}}^{\rm an},\|\ \|_{\tau^{\rm an}})$ par une $(1,1)$-forme lisse $\hat\omega$ qui l'approche au sens 
des courants (on a observ\'e que cela \'etait possible puisque la fonction $-d'd'' \log \|\tau^{\rm an}\|_{\pi^*( E^{\rm an}_X)}$ est une fonction globalement 
psh-approchable au voisinage de tout point $\hat x$ o\`u $\sigma$ n'est pas inversible), il r\'esulte du lemme \ref{lemmeintersection} que l'on a pour tout $1\leq k \leq n$,
$$
\lim\limits_{\lambda_k \rightarrow 0_+} \dots
\lim\limits_{\lambda_1 \rightarrow 0_+}
M_k^{s,\underline\lambda} = \pi_* \Big(\Big(\dots \pi^* (M^{s,\lambda}_k)_{\lambda_1=0}\dots \Big)_{\lambda_k =0}\Big) =
\pi_* \Big([{\rm div}(\sigma^{\rm an})] \wedge \hat\omega^{\wedge^{k-1}}\Big).
$$
On d\'eduit le r\'esultat de la Proposition \ref{propsegre} en approchant
au sens des courants (au fur et \`a mesure que les $\lambda_k$ tendent successivement vers $0$) la forme
$-c_1(\widehat{\mathscr{L}}^{\rm an},\|\ \|_{\tau^{\rm an}})$ par une $(1,1)$-forme lisse
$\hat \omega$.
\end{proof}

\section{Nombres ou cycles de Lelong dans le contexte non archim\'edien}\label{sectlelong}

Soit $\mathscr X$ un espace analytique complexe de dimension $n$ et
$T$ un $(k,k)$-courant positif sur $\mathscr X$. Soit $x_0\in \mathscr X$.
Le nombre de Lelong (ordinaire) $\nu(T,x_0)$ du courant $T$ au point $x_0$ est d\'efini comme la limite lorsque $\epsilon$ tend vers $0^+$ de la fonction croissante 
sur $]0,\epsilon_0]$ (avec
$0<\epsilon_0<<1$)~:
$$
\epsilon \mapsto \frac{1}{\epsilon^{2(n-k)}}\int_{\|x-x_0\|<\epsilon} T \wedge (dd^c \|x-x_0\|^2)^{\wedge^{n-k}}.
$$
Appelons cycle g\'en\'eralis\'e de $\mathscr X$ tout courant de la forme $\pi_* (c)$, o\`u
$\pi~: \mathscr Y \rightarrow \mathscr X$ est un morphisme propre entre espaces analytiques complexes et $c$ est un produit de composantes de formes de Chern 
lisses sur $\mathscr Y$, chacune attach\'ee \`a un  fibr\'e holomorphe $(F\rightarrow \mathscr Y,\|\ \|)$ \'equip\'e d'une m\'etrique lisse~; tel est le cas par 
exemple des courants 
$$
\pi_*\big([Y_\iota]\wedge (-c_1(\hat L, \|\ \|_\tau))^{\wedge^{k-1}}\big) = 
(\pi \circ i_\iota)_* \Big( - \big(c_1(\hat L_{|Y_\iota},(\|\ \|_\tau)_{|Y_\iota})\big)^{\wedge^{k-1}}
\Big)   
$$
($k=1,...,n$), o\`u $Y_\iota$ d\'esigne l'une des composantes 
irr\'eductibles du diviseur exceptionnel $[D]$ de l'\'eclatement $\pi~: \hat{\mathscr X} \rightarrow \mathscr X$ le long du faisceau d'id\'eaux de 
$\mathcal O_{\mathscr X}$ attach\'e \`a une section $s$ d'un fibr\'e hermitien $E_{\mathscr X} \rightarrow \mathscr X$ et $i_\iota~: Y_\iota \rightarrow \hat
{\mathscr{X}}$ l'immersion de $Y_\iota$ dans 
$\hat{\mathscr{X}}$~; la m\'etrique $\|\ \|_\tau$ sur 
le fibr\'e en droites $\hat L=\mathcal O(-[D])$
est ici d\'efinie par $\|\sigma\|_\tau = \|\pi^* s\|_{\pi^* (E_{\mathscr X})}$.
\'Etant donn\'e un point $x_0$ de $\mathscr X$ et un cycle g\'en\'eralis\'e $T$ sur
$\mathscr X$, on sait associer \`a $T$ un nombre de Lelong $\nu(T,x_0)\in \Z$ au point $x_0$.
Par exemple, le nombre
de Lelong $\nu(T_\iota,x_0)$ du courant $T_\iota = \pi_*\big([Y_\iota]\wedge \big(-c_1(\hat L, \|\ \|_\tau)\big)^{\wedge^{k-1}}\big)$ au point $x_0$
s'exprime ainsi lorsque $\xi_{x_0} = \xi_{x_0,0},...,\xi_{x_0,m_{x_0}}$ d\'esigne un syst\`eme de g\'en\'erateurs de l'id\'eal maximal $\EuFrak M_{x_0}$ de 
$\mathscr O_{\mathscr X,x_0}$~:
\begin{equation}\label{nblelongcmplx}
\begin{split}
& \Big[ \cdots \Big[ \int\limits_{\big(\P^{m_{x_0}}_\C\big)^{\nu}}  \Big(\bigwedge\limits_{j=1}^{\nu}  
\big(c_1(\mathcal O_{\P^{m_{x_0}}_\C}(1),\|\ \|_{\rm fs })\big)^{\wedge^{m_{x_0}}}(\kappa_j)\Big) \wedge \\
& \bigwedge\limits_{j=1}^{\nu}
\Big(1-|\langle \kappa_j,\xi_{x_0}\rangle|_{\rm fs}^{2\lambda_j}
+ dd^c \Big( \frac{|\langle \kappa_j,\xi_{x_0}\rangle|^{2\lambda_j}_{\rm fs}}{\lambda_j}\Big)\Big)
\wedge T_\iota \Big]
_{\lambda_1 =0}\cdots \Big]_{\lambda_{\nu}=0} = \nu(T_\iota,x_0)\, [\{x_0\}]
\end{split}
\end{equation}
o\`u $\nu= \min (n+1,m_{x_0}+1)$ et $|\langle \kappa,\xi\rangle|_{\rm fs} := |\langle \kappa,\xi\rangle|/\|\kappa\|$
si $\kappa = [\kappa_0:\dots : \kappa_{m_{x_0}}]$ et $\|\kappa\|$ d\'esigne la norme euclidienne
dans $\C^{m_{x_0}+1}$ (voir la Proposition 5.3 de \cite{ASWY14})~; la notation
$\big[\dots \big]_{\lambda_j=0}$ signifie ici que l'on prolonge m\'eromorphiquement la fonction
holomorphe (\`a valeurs courants) de $\lambda_j$ (pour ${\rm Re}\, \lambda_j >>1$) enserr\'ee par les crochets et que l'on
\'evalue ensuite le coefficient de $\lambda_j^0$ dans le d\'eveloppement en s\'erie de Laurent de ce prolongement m\'eromorphe au voisinage de l'origine.
\vskip 1mm
\noindent
Soit maintenant $X$ une vari\'et\'e alg\'ebrique projective de dimension $n$ d\'efinie sur un corps valu\'e $\K$ et $X^{\rm an}$ 
son analytification au sens de Berkovich. Consid\'erons un courant $T$ sur $X^{\rm an}$ de la forme 
$T = \sum_{\iota,\iota'} (\pi_\iota) _*[\omega_{\iota'}]$, o\`u $\pi_\iota~: Y_\iota^{\rm an} \rightarrow X^{\rm an}$ 
est un morphisme analytique entre analytifi\'es au sens de Berkovich
de vari\'et\'es alg\'ebriques projectives d\'efinies sur $\K$ et $\omega_{\iota'}$ est un produit de premi\`eres formes de Chern de 
fibr\'es en droites $(\mathscr{L}_{\iota,\iota'}^{\rm an},\|\ \|_{\iota,\iota'}^{\rm an})$, o\`u
$\|\ \|_{\iota,\iota'}^{\rm an}$ est une m\'etrique formelle PL globalement psh approchable sur le fibr\'e 
$\mathscr{L}^{\rm an}_{\iota,\iota'} \rightarrow  Y_\iota^{\rm an}$.
Si $x_0 $ est un point ferm\'e de  $X$, on peut analytifier le morphisme $\iota_{x_0}~: \{x_0\} \rightarrow X$ et
consid\'erer $\{x_0\}^{\rm an}$ comme un sous-ensemble de Zariski de dimension $0$ de
$X^{\rm an}$. Soit $\xi_{x_0} = (\xi_{x_0,0},...,\xi_{x_0,m_{x_0}})$ un
syst\`eme de g\'en\'erateurs de l'id\'eal maximal $\EuFrak M_{x_0}$ de $\mathscr O_{X,x_0}$ et
$\nu = \min(n+1,m_{x_0}+1)$. On consid\`ere le fibr\'e $\big(\mathcal O_{\P^{m_{x_0}}_\K}(1)\big)^{\rm an}
\otimes \K^{\rm an}$ sur $\big(\P^{m_{x_0}}_\K \times \mathscr U\big)^{\rm an}$ ($\mathscr U$ ouvert affine contenant $x_0$)
et on analytifie la section $(\kappa,x)\mapsto \langle \kappa,\xi_{x_0}(x)\rangle$ en une
section du fibr\'e $\big(\mathcal O_{\P^{m_{x_0}}_\K}(1)\big)^{\rm an}
\otimes \K^{\rm an}$ au-dessus de $\big(\P^{m_{x_0}}_\K\times \mathscr U\big)^{\rm an} =
\big(\P^{m_{x_0}}_\K\big)^{\rm an} \times U$. On choisit une m\'etrique semi-positive sur
$\mathcal O_{\P^{m_{x_0}}_\K}(1)$ induisant une m\'etrique PL sur
$\big(\mathcal O_{\P^{m_{x_0}}_\K}(1)\big)^{\rm an}$ que l'on note $\|\ \|_{\rm moy}$ et pour laquelle la mesure de Monge-Amp\`ere 
$\big(c_1(\mathcal O_{\P^{m_{x_0}}_\K}(1),\|\ \|_{\rm moy})\big)^{\wedge^{m_{x_0}}}$ est une mesure atomique 
(par exemple la mesure de Dirac au point de Gau\ss \ lorsque $\|\ \|_{\rm moy}$ est la m\'etrique induite par le choix de 
la m\'etrique standard sur $\mathcal O_{\P^{m_{x_0}}_\K}(1)$). On d\'efinit ainsi un courant sur $X^{\rm an}$ (en s'inspirant
de l'approche \eqref{nblelongcmplx}) de support le sous-ensemble de Zariski $\{x_0\}^{\rm an}$~:
\begin{equation}
\begin{split}
& \lim\limits_{\lambda_\nu \rightarrow 0}
\Big(\lim\limits_{\lambda_{\nu-1} \rightarrow 0} \Big(
\cdots \Big(\lim\limits_{\lambda_1 \rightarrow 0} \, \int\limits_{\big((\P^{m_{x_0}}_\K)^{\rm an}\big)^{\nu} } \Big(\bigwedge\limits_{j=1}^\nu  
\big(c_1(\mathcal O_{\P^{m_{x_0}}_\K}(1),\|\ \|_{\rm moy})\big)^{\wedge^{m_{x_0}}}(\kappa_j)\Big) \wedge \\
& \bigwedge\limits_{j=1}^\nu
\Big([1 - \|\langle \kappa_j,\xi_{x_0}^{\rm an}\rangle\|_{{\rm moy}}^{\lambda_j}]
+ d'd''\Big(\Big[ \frac{\|\langle \kappa_j,\xi_{x_0}^{\rm an}\rangle\|_{{\rm moy}}^{\lambda_j}}{\lambda_j}\Big]\Big)\Big)\wedge T(x)\Big)
\cdots \Big)\Big).
\end{split}
\end{equation}
Lorsque l'on choisit comme m\'etrique la m\'etrique standard sur $\mathcal O_{\P^{m_{x_0}}_\K}(1)$, le courant ainsi construit est 
ind\'ependant du choix du syst\`eme g\'en\'erateur $\xi_{x_0}$ de l'id\'eal maximal~: si l'on dispose de deux syst\`emes de g\'en\'erateurs
$\xi_{x_0}$ et $\tilde \xi_{x_0}$ pour l'id\'eal maximal $\EuFrak M_{x_0}$, on peut les compl\'eter par des fonctions nulles pour en faire 
deux syst\`emes de g\'en\'erateurs de la m\^eme longueur $m_{x_0} + \tilde m_{x_0}$ et on compare les deux courants construits en
utilisant la m\'etrique PL induite par la m\'etrique standard sur $\mathcal O_{\P^{m_{x_0}
+ \tilde m_{x_0} -1} _\K}(1)$.
Le courant ainsi construit correspond \`a un cycle analytique de dimension pure $0$, de support $\{x_0\}^{\rm an}$ que l'on peut appeler 
cycle de Lelong du courant $T$ sur le $\K$-espace analytique $\{x_0\}^{\rm an}$.

\section{La formule de King dans le contexte non archim\'edien}\label{sectionKing}

Soit (comme dans la section \ref{sectionsegre}) $X$ une vari\'et\'e alg\'ebrique projective de dimension $n$ d\'efinie sur un corps valu\'e $\K$ et
$X^{\rm an}$ son analytification au sens de Berkovich. On consid\`ere un fibr\'e
alg\'ebrique $E_X \rightarrow X$ de rang fini au-dessus de $X$ et on \'equipe son analytifi\'e $ E_X^{\rm an} \rightarrow X^{\rm an}$  
d'une m\'etrique formelle PL \cite[d\'efinition 6.2.9]{ChLD}, que l'on supposera ici globalement psh approchable 
not\'ee $\|\ \|_{E_X^{\rm an}}$
au-dessus de l'analytification $X^{\rm an}$. Soit $s\in \mathcal O_X(E_X)$ une section globale de
$E_X$ dont on notera $s^{\rm an}~: X^{\rm an} \rightarrow E_X^{\rm an}$ l'analytification.
\vskip 1mm
\noindent
Soit $\pi~: \widehat X \longmapsto X$ l'\'eclatement normalis\'e de $X$
via le faisceau coh\'erent d'id\'eaux attach\'e \`a la section globale
$s\in \mathcal O_X(E_X)$ et $\pi^{\rm an}~:\widehat{X}^{\rm an}
\rightarrow X^{\rm an}$ son analytification.
\vskip 1mm
\noindent
Pour chaque $k=0,...,n$, on note $(Y_{k,\iota_k})_{\iota_k}$ la liste des composantes exceptionnelles de l'\'eclatement normalis\'e $\pi~: \widehat X \longmapsto X$
telles que
${\rm codim}_X\, \pi(Y_{k,\iota_k}) = k$ et $( Y_{k,\iota_k}^{\rm an}\hookrightarrow \widehat{X}^{\rm an})_{\iota_k}$ 
la liste de leurs analytifications au sens de Berkovich. On introduit \'egalement l'analytifi\'e $\widehat{\mathscr{L}}^{\rm an}$ 
induit au-dessus de $\widehat{X}^{\rm an}$ par le fibr\'e $L_{\widehat X}$ correspondant au diviseur exceptionnel de l'\'eclatement $\pi$. Ce fibr\'e
$\widehat {\mathscr{L}}^{\rm an}$ est \'equip\'e de la m\'etrique $\|\ \|_{\tau^{\rm an}}$ induite par la m\'etrique
d\'efinie par $\|\sigma\| = \|\pi^*(s)\|_{\pi^*(E_X)}$ si $s = \sigma \otimes \tau$, o\`u
$\sigma$ est une section de $L_{\widehat X}$ et $\tau$ une section ne s'annulant pas de
$L_{\widehat X}^{-1} \otimes \pi^*(E)$.
\\
Pour chaque paire d'entiers $k,\ell\in \{1,...,n\}$, pour chaque indice $\iota_\ell$, on introduit le courant
$T_{k,\ell,\iota_\ell} := \pi_*^{\rm an} \Big([Y_{\ell,\iota_\ell}^{\rm an}] \wedge \big(-c_1(\widehat{\mathscr{L}}^{\rm an},\|\ \|_{\tau^{\rm an}})\big)^{\wedge^{k-1}}\Big)$.
Le support de ce courant est inclus dans l'union des ensembles de Zariski
$\pi^{\rm an}(Y_{\ell,\iota_\ell}^{\rm an})$, sous-ensemble analytique ferm\'e de $X^{\rm an}$ de codimension $\ell$.
\vskip 1mm
\noindent
Lorsque $\ell >k$ et que $\omega \in \mathscr{A}_c^{n-k,n-k}(X^{\rm an})$, on a $(j_{\iota_\ell}^{\rm an})^* \omega =0$ si
$$
j_{\iota_\ell}^{\rm an}~: Y_{\ell,\iota_\ell}^{\rm an} \rightarrow X^{\rm an}
$$
d\'esigne l'analytification du morphisme
$$
Y_{\ell,\iota_\ell} \hookrightarrow \widehat X \stackrel{\pi}{\longrightarrow} X
$$
(pour des raisons de dimension, du fait
que ${\rm codim}_X (\pi(Y_{\ell,\iota_\ell}))=\ell >k$). Il en r\'esulte donc que, d\`es que $\ell >k$, on a $T_{k,\ell,\iota_\ell}=0$ pour tout indice $\iota_\ell$.
\vskip 1mm
\noindent
On remarque aussi que si $\ell < k$, le cycle de Lelong du courant $T_{k,\ell,\iota_\ell}$ en $\{x_0\}^{\rm an}$ dans  $X^{\rm an}$ est le cycle nul (pour tout $x_0\in X$). 
On raisonne pour cela ainsi, apr\`es avoir dans un premier temps approch\'e le $(1,1)$-courant
$-c_1(\widehat{\mathscr{L}}^{\rm an},\|\ \|_{\tau^{\rm an}})$ par une suite
de $(1,1)$-formes de Chern lisses en utilisant le fait que la m\'etrique PL en jeu ici est suppos\'ee globalement psh approchable.
\begin{itemize}
\item On multiplie le courant $T_{k,\ell,\iota_\ell}$
par le \og courant moyen\fg\ (on rappelle que le courant 
$\big(c_1(\mathcal O_{\P^{m_{x_0}}_\K}(1),\|\ \|_{\rm moy})\big)^{\wedge^{m_{x_0}}}(\kappa_1)$ correspond \`a une mesure atomique)
$$
\int\limits_{\big(\P^{m_{x_0}}_\K\big)^{\rm an}}
\big(c_1(\mathcal O_{\P^{m_{x_0}}_\K}(1),\|\ \|_{\rm moy})\big)^{\wedge^{m_{x_0}}}(\kappa_1)
\wedge \Big(
[1 - \|\langle \kappa_1,\xi_{x_0}^{\rm an}\rangle\|_{{\rm moy}}^{\lambda_1}]
+ d'd''\Big( 
\Big[\frac{\|\langle \kappa_1,\xi_{x_0}^{\rm an}\rangle\|_{{\rm moy}}^{\lambda_1}}{\lambda_1}\Big]\Big)\Big).
$$
En utilisant le fait que le support de toute forme $\varphi\in \mathscr{A}_c^{p,n-1}( Y_{\ell,\iota_\ell}^{\rm an})$ ($0\leq \ell\leq n-1$) 
ne saurait intersecter aucun sous-ensemble de Zariski propre de $ Y_{\ell,\iota_\ell}^{\rm an}$
(on applique \`a nouveau \cite{ChL}, 5.1), on voit
que soit le courant obtenu ainsi est nul, soit l'analytifi\'e de
$\P^{m_{x_0}}_\K \times \pi(Y_{\iota_\ell,\ell})$ dans
$\P^{m_{x_0}}_\K \times\mathscr U$ (on reprend ici les notations utilis\'ees
dans la section \ref{sectlelong})
est inclus dans $\{\langle \kappa_1,\xi_{x_0}^{\rm an}\rangle =0\}$ pour un $\kappa_1$ g\'en\'erique 
(la moyennisation effectu\'ee ici correspond 
\`a la prise de mesure de Dirac au point de Gauss).
\item On r\'eit\`ere si n\'ecessaire (lorsque $\ell<k-1$) cette op\'eration
$k-\ell -1$ fois. Cette op\'eration ne saurait se poursuivre sans que l'on ne rencontre lors du processus le courant nul.
\end{itemize}
\vskip 1mm
\noindent
Ainsi l'on peut \'ecrire, pour tout
$k\in [{\rm codim}_X s^{-1}(0),n]$,
\begin{equation}\label{king1}
\begin{split}
& \pi_*^{\rm an} \Big(
[{\rm div}(\sigma^{\rm an})]
\wedge  \big(-c_1(\widehat{\mathscr{L}}^{\rm an},\|\ \|_{\tau^{\rm an}})\big)^{\wedge^{k-1}}\Big)
= \\
& = \sum_{\iota_k}
\pi_*^{\rm an} \Big(
[Y^{\rm an}_{k,\iota_k}]
\wedge  \big(-c_1(\widehat{\mathscr{L}}^{\rm an},\|\ \|_{\tau^{\rm an}})\big)^{\wedge^{k-1}}\Big) + \mathscr N_k[s]\,,
\end{split}
\end{equation}
de mani\`ere \`a ce que le sous-ensemble des points $\{x_0\}^{\rm an}$ de $X^{\rm an}$ o\`u le $(k,k)$-courant $\mathscr N_k[s]$ a 
un cycle de Lelong non nul soit de codimension au moins \'egale \`a
$k+1$.
\vskip 1mm
\noindent
On peut donc \'enoncer la version suivante du Th\'eor\`eme de King, dans le cadre cette fois non archim\'edien. 
Ce r\'esultat constitue le pendant du Th\'eor\`eme 1.1 de \cite{ASWY14}.
Nous ne donnerons l'\'enonc\'e ici que dans le contexte alg\'ebrique, contexte o\`u nous nous pla\c cons dans cet article. 
La terminologie \og stable\fg\ et \og mobile\fg\ fait ici r\'ef\'erence \`a celle classiquement introduite dans le cadre de la th\'eorie de 
l'intersection impropre en g\'eom\'etrie analytique complexe, voir par exemple l'introduction de \cite{ASWY14} ainsi que \cite{GafGas} 
o\`u cette terminologie est introduite. 

\begin{theorem}
Soit $X$ une vari\'et\'e alg\'ebrique projective de dimension $n$ d\'efinie sur un corps valu\'e $\K$ et
$X^{\rm an}$ son analytification au sens de Berkovich. On consid\`ere un fibr\'e
alg\'ebrique $E_X \rightarrow X$ de rang fini au-dessus de $X$, l'on suppose que le fibr\'e $E_X^{\rm an} \rightarrow X^{\rm an}$  
est \'equip\'e d'une m\'etrique
formelle PL, not\'ee $\|\ \|_{E_X^{\rm an}}$, au-dessus de l'analytification
$X^{\rm an}$. Soit $s\in \mathcal O_X(E_X)$ une section globale de
$E_X$ et $s^{\rm an}\in \mathcal O_{X^{\rm an}} (E_X^{\rm an})$ son analytification. 
Pour tout $k=0,...,n$, on note $(Y_{k,\iota_k})_{\iota_k}$ la liste des composantes exceptionnelles de 
l'\'eclatement normalis\'e $\pi~: \widehat X \longmapsto X$ (le long du faisceau coh\'erent d'id\'eaux 
attach\'e \`a la section $s$) telles que
${\rm codim}_X\,  \pi(Y_{k,\iota_k}) = k$ et $(Y_{k,\iota_k}^{\rm an}\hookrightarrow \widehat{X}^{\rm an})_{\iota_k}$ 
la liste de leurs analytifications au sens de Berkovich. La composante de bidegr\'e $(k,k)$ du courant $M^s$ de Segre
se scinde, pour $k=1,...,n$ en sa composante \og stable\fg~:
$$
(M^s_k)_{\rm stable} =  \sum_{\iota_k}
\pi_*^{\rm an} \Big(
[Y^{\rm an}_{k,\iota_k}]
\wedge  \big(-c_1(\widehat{\mathscr{L}}^{\rm an},\|\ \|_{\tau^{\rm an}})\big)^{\wedge^{k-1}}\Big)
$$
et sa composante \og mobile\fg~:
$$
(M^s_k)_{\rm mobile} =  \sum_{\ell =0}^{k-1}\sum\limits_{\iota_\ell}
\pi_*^{\rm an} \Big(
[Y^{\rm an}_{\ell,\iota_\ell}]
\wedge  \big(-c_1(\widehat{\mathscr{L}}^{\rm an},\|\ \|_{\tau^{\rm an}})\big)^{\wedge^{k-1}}\Big)
$$
telle que, pour tout point ferm\'e $x\in X$, le cycle de Lelong du courant
$(M^s_k)_{\rm mobile}$ sur $\{x\}^{\rm an}$ soit nul.
\end{theorem}

\begin{proof}
Supposons que $E_X$ soit de rang $m+1$.
Soit $x^{\rm an}$ un point de $X^{\rm an}$ et 
$U_{x^{\rm an}}$ un domaine analytique contenant $x^{\rm an}$ au-dessus duquel $E^{\rm an}$ admette un rep\`ere orthonorm\'e 
$\{e_0,...,e_m\}$. La section $s^{\rm an}$ s'exprime dans $U_{x^{\rm an}}$ sous la forme
$$
s^{\rm an} = \sum\limits_{\ell=0}^m s_{\ell}^{\rm an}\, e_j,
$$
o\`u les fonctions coordonn\'ees $s_{\ell}^{\rm an}$, $\ell =0,...,m$, sont des fonctions analytiques
et o\`u
$$
\|s\| = \max\limits_{0\leq \ell\leq m} |s_\ell^{\rm an}|.
$$
Auquel cas, on peut consid\'erer, au lieu de la factorisation $(\pi^{\rm an})^*(s)=\sigma^{\rm an} \otimes \tau^{\rm an}$ 
(o\`u $\sigma^{\rm an}$ est une section du fibr\'e $\widehat{\mathscr{L}}^{\rm an}$), ind\'ependamment chaque factorisation 
$(\pi^{\rm an})^* (s_\ell^{\rm an}) = \sigma^{\rm an}
\otimes \tau^{\rm an}_\ell$, les $\tau^{\rm an}_\ell$ ($\ell=0,...,m$)
\'etant des sections au-dessus de $(\pi^{\rm an})^{-1}(U_{x^{\rm an}})$
du fibr\'e $(\widehat{\mathscr{L}}^{\rm an})^{-1}$. Reprenant la construction
des courants de Vogel telle qu'elle a \'et\'e d\'ecrite dans la section
\ref{sectionVogel}, on observe que, pour tout $k=1,...,n$, pour tout $\iota_k$, on peut
construire \`a l'aide du Th\'eor\`eme \ref{theorempfonctions2} un $(k-1,k-1)$-courant $A_k
\in \mathscr{D}_{n-(k-1),n-(k-1)}(\pi^{\rm an}(U_{x^{\rm an}}))$
de support inclus dans l'ensemble de Zariski $\pi^{\rm an}(Y_{k,\iota_k}^{\rm an})$ 
(de codimension $k$ dans $X^{\rm an}$, donc dans $U_{x^{\rm an}}$), solution de l'\'equation de Green \og moyennis\'ee\fg \
\begin{eqnarray*}
d'd''A_k &=& \lim\limits_{\lambda_{k-1} \rightarrow 0}
\Big(
\cdots \Big(\lim\limits_{\lambda_1 \rightarrow 0} \, \int\limits_{\big((\P^m_\K)^{\rm an}\big)^{k-1}} 
\Big(\bigwedge\limits_{j=1}^{k-1} \big(c_1(\mathcal O_{\P^m_\K}(1),\|\ \|_{\rm moy})\big)^{\wedge^m}(\kappa_j)\Big) \wedge \\
& & \bigwedge_{j=1}^{k-1} d'd''\Big(\Big[\frac{\|\langle \kappa_j,\tau^{\rm an}\rangle\|_{(\widehat
{\mathscr{L}}^{\rm an})^{-1}
,{\rm moy}}^{\lambda_j}}{\lambda_j}\Big]\Big) \wedge
[Y_{k,\iota_k}^{\rm an}]\\
& & \qquad \qquad  - [Y_{k,\iota_k}^{\rm an}]\wedge  \big(-c_1(\widehat{\mathscr{L}}^{\rm an},\|\ \|_{\tau^{\rm an}})\big)^{\wedge^{k-1}}\Big)\Big)
\end{eqnarray*}
Chaque courant $\pi^{\rm an}_* (A_{k,\iota_k})\in \mathscr D_{n-(k-1),n-(k-1)}(U_{x^{\rm an}})$ est de support inclus dans 
l'ensemble de Zariski $\pi^{\rm an}(Y_{k,\iota_k}^{\rm an})$~; un tel courant, de part sa construction m\^eme via le 
prolongement analytique, est donc nul pour des raisons de dimension et la composante stable  $(M^s_k)_{\rm stable}$ 
de la composante $M^s_k$ du courant de Segre $M^s$ s'exprime donc aussi
comme
\begin{eqnarray*}
& & (M^s_k)_{\rm stable} = \\
& & \pi^{\rm an}_*\Big(
\sum\limits_{\iota_k} \Big( \lim\limits_{\lambda_{k-1} \rightarrow 0}
\Big(
\cdots \Big(\lim\limits_{\lambda_1 \rightarrow 0} \,  \int\limits_{\big((\P^m_\K)^{\rm an}\big)^{k-1}} 
\Big(\bigwedge\limits_{j=1}^{k-1} \big(c_1(\mathcal O_{\P^m(\K)}(1),\|\ \|_{\rm moy})\big)^{\wedge^m}(\kappa_j)\Big) \wedge \\
& &
\bigwedge_{j=1}^{k-1} d'd''\Big( \Big[\frac{\|\langle \kappa_j,\tau^{\rm an}\rangle\|_{(\widehat
{\mathscr{L}}^{\rm an})^{-1}
,{\rm moy}}^{\lambda_j}}{\lambda_j}\Big]\Big)\Big) \wedge
[Y_{k,\iota_k}^{\rm an}]\Big)\Big)\Big).
\end{eqnarray*}
Lorsque $x$ est un point ferm\'e de $X$,
le cycle de Lelong de $M_k^s$ en $x^{\rm an}$ (qui est aussi celui de
$(M^s_k)_{\rm stable}$) s'interpr\`ete donc comme un courant de Vogel (au sens introduit dans la section \ref{sectionVogel}), 
ce de mani\`ere analogue \`a ce qui se produit dans le cadre archim\'edien (voir les sections 7 et 8 de
\cite{ASWY14}).
\end{proof}
\vskip 3mm
\noindent
\dedicatory{{\bf Remerciements}
\vskip 1mm
\noindent
Je tiens \`a remercier le rapporteur d'avoir lu extr\^ement attentivement les diff\'erentes versions de ce document et de m'avoir propos\'e de nombreuses   remarques et suggestions qui ont grandement contribu\'e \`a l'am\'eliorer. \\
L'auteur tient aussi \`a exprimer sa profonde gratitude \`a Alain Yger, Professeur \`a l'Institut de math\'ematiques de Bordeaux (Universit\'e de Bordeaux,  France), pour son aide tant pr\'ecieuse lors de la recherche.  Il lui est \'egalement agr\'eable de remercier Salomon Sambou, Professeur  de l'Universit\'e Assane Seck de Ziguinchor (S\'en\'egal), pour des discussions int\'eressantes.
}

\end{document}